\documentclass[12pt,twoside,singlespace]{amsart}
\usepackage{array, paralist, enumerate, amsmath, amsfonts, amssymb, amscd, color, mathrsfs,comment}
\usepackage{amsthm} 
\usepackage{bbm}
\usepackage{changes}
\definechangesauthor[name={Yan}, color=red]{YZ}
\usepackage{geometry}
\usepackage{framed}
\usepackage{hyperref}
\usepackage{graphicx}
\usepackage{epstopdf}
\usepackage[all,cmtip]{xy}
\usepackage{tikz}
\usepackage{tkz-graph}
\usetikzlibrary{arrows,%
                shapes,positioning}
\usetikzlibrary{backgrounds}
\definecolor{DarkBlue}{rgb}{0,0,0.8} 
\definecolor{DarkGreen}{rgb}{0,0.5,0.0} 
\definecolor{DarkRed}{rgb}{0.9,0.0,0.0} 

\usepackage[T1]{fontenc}
\usepackage[latin1]{inputenc}

\newtheorem*{thm*}{Theorem}

\numberwithin{equation}{section}

\newtheorem{lem}[equation]{Lemma}
\newtheorem{cor}[equation]{Corollary}
\newtheorem{prop}[equation]{Proposition}

\theoremstyle{definition}

\newtheorem{ex}[equation]{Example}	
\newtheorem{remark}[equation]{Remark}

\newcommand{\ZZ}{\mathbb{Z}}

\newcommand{\gen}[1]{\langle #1\rangle}
\newcommand{\ophi}{\phi}

\DeclareMathOperator{\tr}{tr}

\newcommand{\K}{\overline{K_4}}

\usepackage{color}
\usepackage{ifthen}

\newboolean{usecolor}
\setboolean{usecolor}{true}
\ifthenelse{\boolean{usecolor}}
{
  \definecolor{colore}{cmyk}{0,1,0.6,0}
  \definecolor{coloresimo}{cmyk}{1,0.6,0,0}
}
{
  \definecolor{colore}{cmyk}{0,0,0,1}
  \definecolor{coloresimo}{cmyk}{0,0,0,1}
}

\title{Enumerative Gadget Phenomena for $(4,1)$-Adinkras}
\author{Isaac Friend}
\address{Isaac Friend,
	Dept. of Physics,
	Brown University,
	Providence, RI, USA}
\email{icfriend@uchicago.edu}
\author{Jordan Kostiuk}
\address{Jordan Kostiuk, 
	Dept. of Mathematics,
	Brown University,
	Providence, RI, USA}
\email{jordan\_kostiuk@brown.edu}
\author{Yan X Zhang}
\address{Yan X Zhang,
Dept. of Mathematics and Statistics,
San Jose State University,
San Jose, CA, USA}
\email{yan.x.zhang@math.sjsu.edu}
\begin{document}

\pagestyle{plain}

\begin{abstract}
Adinkras are combinatorial objects developed to study supersymmetry representations. Gates et al. introduced the \emph{gadget} as a function of pairs of adinkras, obtaining some mysterious results for $(n=4, k=1)$ adinkras with computer-aided computation. Specifically, very few values of the gadget actually appear, suggesting a great deal of symmetry in these objects. In this paper, we compute gadgets symbolically and explain some of these observed phenomena with group theory and combinatorics. Guided by this work, we give some suggestions for generalizations of the gadget to other values of the $n$ and $k$ parameters.
\end{abstract}

\maketitle

\section{Introduction}

In this paper, we are motivated by the $(n=4, k=1)$ \emph{adinkra} which has been the principal object of interest in several adinkra papers, such as \cite{calkinsAdinkras0branesHoloraumy2015,gatesLorentzCovariantHoloraumyInduced2015,gates:genomics,gatesjr.AdinkrasOrderedQuartets2017}, with adinkras defined introduced generally in \cite{d2l:first}. 
The $(4, 1)$ adinkras are bipartite graphs with four white vertices (bosons) and four black vertices (fermions), together with extra decorations encoding physical data. 
A $(4,1)$ adinkra encodes the information of a representation of the $(4,1)$ supersymmetry algebra.

Continuing from observations made in \cite{douglas,douglasAutomorphismPropertiesClassification2015}, Gates introduced a map in \cite{gatesLorentzCovariantHoloraumyInduced2015} called the \emph{gadget} which takes as input a pair of adinkras (actually from a special family called \emph{valise adinkras}) and outputs a number. For each adinkra $k \in \{1,2\}$, the procedure creates the adinkra's \emph{(fermionic) holoraumy matrices} $\widetilde{V}_{IJ}$ and $\widetilde{V}'_{IJ}$, each adinkra's set having one matrix for each $I \neq J$ with $I, J \in \{1,2,3,4\}$. Finally, the \emph{gadget} is defined in the special case for $(4,1)$ as 
\[\mathcal{OG}(A,A')= \frac{1}{48}\sum_{I \neq J}\tr(\widetilde{V}_{IJ}\widetilde{V}'_{IJ}).\]

We know (from, e.g., \cite{zhang:adinkras}) that there are $36864$ $(4,1)$ (valise) adinkras. Gates et al. \cite{gatesjr.AdinkrasOrderedQuartets2017} discovered the strange fact that the function $\mathcal{OG}$ has a range of cardinality only $4$ when all $36864^2$ pairs $(A,A')$ of $(4,1)$ adinkras are used as input: $\{1, 1/3, -1/3, 0\}$. Furthermore, there are some nice algebraic properties; for example, if the two adinkras have different \emph{chiralities} (an invariant associated with a $(4,1)$ adinkra), the gadget would evaluate to $0$. 

In our paper, we explain these phenomena with elementary group theory and combinatorics, in a style tailored for future work (e.g. for more general $n$ and $k$). In Section~\ref{sec:prelims}, we cover the background of adinkras, gadgets, and \emph{symmetrized gadgets}---our proposed object to study alongside the gadget. 
In Sections~\ref{sec:4-1} and \ref{sec:gadgets}, we introduce tools to study gadgets, exploiting symmetries particular to $(4,1)$. 
In Section~\ref{sec:yan}, we explain the results of \cite{gatesjr.AdinkrasOrderedQuartets2017} using the introduced machinery. 
In Section~\ref{sec:jordan}, we similarly analyze the symmetrized gadget and highlight its many nice features.
In Section~\ref{sec:gadgets-general}, we give some computational results and thoughts on how these gadgets generalize (or fail to generalize) to higher dimensions.
We end with some discussion in Section~\ref{sec:conclusion}.

\section{Preliminaries and Definitions}
\label{sec:prelims}

In this section, we review the essentials of garden algebras and adinkras, leading up to the definition of \emph{gadgets} in Section~\ref{sec:gadgets}. As this paper mostly pertains to mathematics, we do not focus on the physics background. The interested reader should refer to \cite{d2l:first} for the introduction of adinkras in the physics literature, or \cite{zhang:adinkras} for a more mathematical treatment of adinkras. 

\subsection{Adinkras}

An $(n,k)$ \emph{chromotopology} is a finite connected simple graph $G$ such that:

\begin{itemize}
\item $G$ is $n$-regular (every vertex has exactly $n$ incident edges) and bipartite;
\item Let $V = B \cup F$ be the vertices decomposed into bosons and fermions. We know from \cite{zhang:adinkras}  that $|V|$ is a power of $2$, so $|V| = 2^{n-k}$ determines $|V|$ from $n$ and $k$. 
\item The edges of $G$ are colored by $n$ colors such that every vertex is incident to exactly one edge of each color;
\item We assume the colors come with an ordering; that is, we can label the colors with the integers $1$ through $n$;
\item For any distinct colors $i$ and $j$, the edges in $G$ with colors $i$ and $j$ form a disjoint union of $4$-cycles. 
\end{itemize}

With this structure, a \emph{ranking} of a chromotopology $G$ is a map $h$ from the vertices of $G$ to $\ZZ$ that satisfies certain restraints. 
In this paper, we limit ourselves to the \emph{valise ranking}, which simply means having $h(v) \in \{0,1\}$ for every vertex $v$ and having every edge $(x,y)$ in $G$ satisfy $\{h(x), h(y)\} = \{0,1\}$. 
We visualize this by putting the vertices into two rows, each row corresponding to one of the parts of the bipartition, with the arbitrary choice of fermions $F$ on bottom and bosons $B$ on top.

A \emph{dashing} of a chromotopology $G$ is a map $d$ from the edges of $G$ to $\ZZ_2$ such that the sum of $d(e)$ as $e$ runs over each $2$-colored $4$-cycle (that is, a $4$-cycle of edges using a total of $2$ colors) is $1 \in \ZZ_2$; alternatively, every $2$-colored $4$-cycle contains an odd number of $1$'s. We typically draw a dashed edge for $e$ if $d(e) = 1$ and a solid edge if $d(e) = 0$.

Finally, a \emph{valise adinkra} (although we will just say ``adinkra'' for short as we do not deal with other types of adinkras in this paper) is just a chromotopology with an odd dashing and a valise ranking. See Figure~\ref{fig:initial examples} for an example of an $(n=3, k=0)$ adinkra with $2^{3-0} = 8$ vertices. Whenever we have an adinkra $A$, we denote the underlying chromotopology $C(A)$.

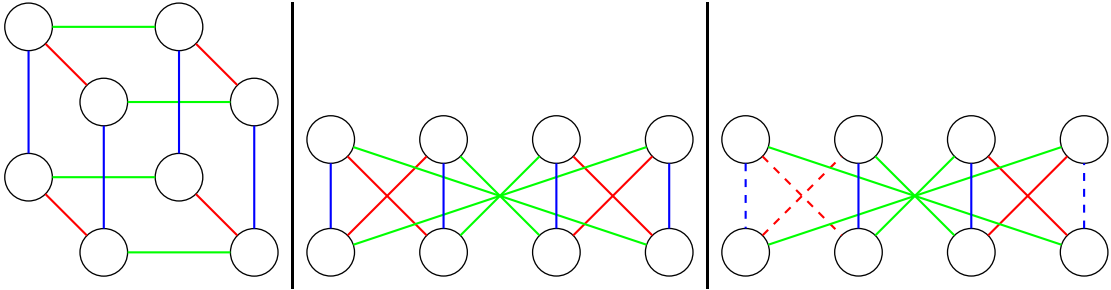
\begin{figure}[htb]
\begin{center}

\begin{tabular}{c|c|c}
\begin{tikzpicture}[scale=0.10]
\SetVertexNoLabel
\SetUpEdge[labelstyle={draw}]
\Vertex[x=0,y=0]{111}
\Vertex[x=20,y=0]{011}
\Vertex[x=0,y=20]{101}
\Vertex[x=20,y=20]{001}
\Vertex[x=-10,y=10]{110}
\Vertex[x=10,y=10]{010}
\Vertex[x=-10,y=30]{100}
\Vertex[x=10,y=30]{000}
\Edge[color=red](100)(101)
\Edge[color=red](000)(001)
\Edge[color=red](010)(011)
\Edge[color=red](110)(111)
\Edge[color=green](000)(100)
\Edge[color=green](001)(101)
\Edge[color=green](010)(110)
\Edge[color=green](011)(111)
\Edge[color=blue](000)(010)
\Edge[color=blue](001)(011)
\Edge[color=blue](100)(110)
\Edge[color=blue](101)(111)
\end{tikzpicture}
&
\begin{tikzpicture}[scale=0.15]
\SetVertexNoLabel
\SetUpEdge[labelstyle={draw}]
\Vertex[x=0,y=20]{111}
\Vertex[x=0,y=10]{101}
\Vertex[x=20,y=20]{010}
\Vertex[x=20,y=10]{000}
\Vertex[x=-10,y=10]{110}
\Vertex[x=10,y=10]{011}
\Vertex[x=-10,y=20]{100}
\Vertex[x=10,y=20]{001}
\Edge[color=red](100)(101)
\Edge[color=red](000)(001)
\Edge[color=red](010)(011)
\Edge[color=red](110)(111)
\Edge[color=green](000)(100)
\Edge[color=green](001)(101)
\Edge[color=green](010)(110)
\Edge[color=green](011)(111)
\Edge[color=blue](000)(010)
\Edge[color=blue](001)(011)
\Edge[color=blue](100)(110)
\Edge[color=blue](101)(111)
\end{tikzpicture}
&
\begin{tikzpicture}[scale=0.15]
\SetVertexNoLabel
\SetUpEdge[labelstyle={draw}]
\Vertex[x=0,y=20]{111}
\Vertex[x=0,y=10]{101}
\Vertex[x=20,y=20]{010}
\Vertex[x=20,y=10]{000}
\Vertex[x=-10,y=10]{110}
\Vertex[x=10,y=10]{011}
\Vertex[x=-10,y=20]{100}
\Vertex[x=10,y=20]{001}
\Edge[color=red,style=dashed](100)(101)
\Edge[color=red](000)(001)
\Edge[color=red](010)(011)
\Edge[color=red,style=dashed](110)(111)
\Edge[color=green](000)(100)
\Edge[color=green](001)(101)
\Edge[color=green](010)(110)
\Edge[color=green](011)(111)
\Edge[color=blue, style=dashed](000)(010)
\Edge[color=blue](001)(011)
\Edge[color=blue, style=dashed](100)(110)
\Edge[color=blue](101)(111)
\end{tikzpicture}
\end{tabular}
\caption{Left to right: a $(n=3, k=0)$ chromotopology (which happens to be the $3$-dimensional Hamming cube), a valise ranked chromotopology, a (valise) adinkra, obtained in sequence by adding more and more structure. \label{fig:initial examples}}
\end{center}
\end{figure}

\subsection{Garden Algebras}

A \emph{permutation matrix} is a square matrix with exactly one $1$ in each row and each column and with $0$'s elsewhere. A \emph{signed permutation matrix} is a permutation matrix, except the $1$'s are also allowed to be $-1$'s. We define a $GR(d,N)$ \emph{garden algebra} to be an algebra generated by $N > 0$ $d \times d$ signed permutation matrices $\{L_1, \ldots, L_N\}$ satisfying the following the relations:
\begin{align*}
  L_I L_J^T + L_J L_I^T & = 2 \delta_{IJ} \mathbbm{1}_d  \\
  L_I^T L_J + L_J^T L_I & = 2 \delta_{IJ} \mathbbm{1}_d,
\end{align*}
where $\mathbbm{1}_d$ denotes the identity $d \times d$ matrix. Given a signed permutation matrix $M$, we use $|M|$ to denote the permutation matrix where all the $-1$'s are changed to $1$'s.

\begin{ex}
\label{ex:matrices}
One possible list $(L_1, L_2, L_3, L_4)$ for a $GR(4,4)$ algebra is:
\[
\begin{bmatrix} 0 & 0 & 1 & 0 \\
0 & 0 & 0 & 1 \\
1 & 0 & 0 & 0 \\
0 & 1 & 0 & 0 \end{bmatrix}, \begin{bmatrix} 0 & -1 & 0 & 0 \\
1 & 0 & 0 & 0 \\
0 & 0 & 0 & -1 \\
0 & 0 & 1 & 0 \end{bmatrix},  \begin{bmatrix} 0 & 0 & 0 & 1 \\
0 & 0 & -1 & 0 \\
0 & -1 & 0 & 0 \\
1 & 0 & 0 & 0 \end{bmatrix}, \begin{bmatrix} 1 & 0 & 0 & 0 \\
0 & 1 & 0 & 0 \\
0 & 0 & -1 & 0 \\
0 & 0 & 0 & -1 \end{bmatrix}.
\]
The list $(|L_1|, |L_2|, |L_3|, |L_4|)$ is
\[
\begin{bmatrix} 0 & 0 & 1 & 0 \\
0 & 0 & 0 & 1 \\
1 & 0 & 0 & 0 \\
0 & 1 & 0 & 0 \end{bmatrix}, \begin{bmatrix} 0 & 1 & 0 & 0 \\
1 & 0 & 0 & 0 \\
0 & 0 & 0 & 1 \\
0 & 0 & 1 & 0 \end{bmatrix}, \begin{bmatrix} 0 & 0 & 0 & 1 \\
0 & 0 & 1 & 0 \\
0 & 1 & 0 & 0 \\
1 & 0 & 0 & 0 \end{bmatrix}, \begin{bmatrix} 1 & 0 & 0 & 0 \\
0 & 1 & 0 & 0 \\
0 & 0 & 1 & 0 \\
0 & 0 & 0 & 1 \end{bmatrix}.
\]
\end{ex}

The main point is that valise adinkras are in bijection with lists of Garden algebra generators (see \cite{zhang:counting} for some nuances in the counting), as the Garden algebra generators are just adjacency matrices for the adinkras.
Indeed, suppose that we label the fermions $f_1, \ldots, f_4$ and bosons $b_1, \ldots, b_4$ and use the colors $\{1,2,3,4\}$. Then, for each color $I$, the corresponding $L$-matrix $L_I$ is the $4\times 4$ matrix whose rows are indexed by the four bosons $B$ and whose columns are indexed by the four fermions $F$, filled in by the following rules:
\begin{itemize}
	\item if boson $b_i$ is adjacent to fermion $f_j$ by the edge of color $I$, then the $(i,j)$ entry of $L_I$ is $\pm 1$; otherwise, set the entry to $0$;
	\item the sign of the $(i,j)$-entry (when non-zero) is determined by the dashing: we have $+1$ if the corresponding edge is solid and $-1$ if the edge is dashed. 
\end{itemize}
Similarly, one can speak of the $R$-matrices $R_I$, which reverse the roles of $B$ and $F$ (e.g. if the fermion $f_i$ is adjacent to boson $b_j$, then the $(i,j)$-th entry is $\pm 1$). This gives the same information, and it is easy to check we have $R_I=L_I^t$. Because of the $2$-colored $4$-cycle rule, these signed permutation matrices satisfy:
\begin{eqnarray*}
	L_IR_J+L_JR_I &=& 2\delta_{IJ}\mathbbm{1}_d\\
	R_IL_J+R_JL_I &=& 2\delta_{IJ}\mathbbm{1}_d,
\end{eqnarray*}
which are the same as the definition of garden algebras earlier.

We will often choose to forget about the signs of garden algebra matrices, obtaining permutation matrices from signed permutation matrices. This has the same effect on a (valise) adinkra $A$ as forgetting about the dashings, which gives the chromotopology $C(A)$.

\subsection{From Adinkras to Linear Operators}
\label{sec:operators}
In this section, we would like to interpret the $L$ and $R$ matrices considered earlier as operators on a vector space. 
Let $V=B\cup F$ denote the bipartition of the vertices and let $\mathbf{C}^V$ denote the $\mathbf{C}$-vector space whose basis elements correspond to the elements of $V$. 
We have a decomposition $\mathbf{C}^V=\mathbf{C}^B\oplus\mathbf{C}^F$. 
Define operators $\ophi_I: \mathbf{C}^V \rightarrow\mathbf{C}^V$ by sending each vertex $v$ to the adjacent vertex along the unique edge of color $I$ incident to $v$, and then multiplying by $\pm 1$ according to the dashing of this edge. 
We adopt the convention that the operator $\phi_I$ acts \emph{on the right}, denoting by $v^{\phi_I}$ the image of $v$ under $\phi_I$.
On $\mathbf{C}^V$, the $\ophi_I$ satisfy relations analogous to the garden-algebra relations: namely,
$$\ophi_I\ophi_J+\ophi_J\ophi_I=2\delta_{IJ},$$
as operators on $\mathbf{C}^V$. 

\begin{lem} 
\label{lem:operator-properties}
Suppose adinkra $A$ has $d = 2^{n-k}$ vertices. We have the following properties for any $I \neq J$.
\begin{enumerate}
  \item $\ophi_I^2 = \mathbbm{1}_{d}$. 
  \item $\ophi_{I}\ophi_J\ophi_I\ophi_J = -\mathbbm{1}_{d}$.
  \item $\ophi_I\ophi_J = -\ophi_J\ophi_I$.
\end{enumerate}
\end{lem}
\begin{proof}
  The proofs are as follows:
  \begin{enumerate}
  \item Every time we follow a color and then go back immediately, we traverse either a solid edge twice or a dashed edge twice. Thus, the total number of dashings is even and we are back at the original vertex.
  \item This is the $2$-color $4$-cycle condition of odd dashings; it means every time we follow color $I$, $J$, $I$, and $J$ again, we get to the same vertex with a total of an odd number of dashings, which means we pick up a $-1$ sign.
    \item We obtain this property by multiplying both sides of the second property on the right by $\ophi_J$, then $\ophi_I$. \qedhere
  \end{enumerate}
  \end{proof}

If we fix an ordering of the bosons and fermions (fermions before bosons), we can represent $\phi_I$ as a matrix $M_{\phi_I}$.
We will often identify $\phi_I$ with the matrix $M_{\phi_I}$ if there is no confusion, i.e., a fixed ordering of the vertices has been chosen.
The matrix $\phi_I$ takes the form $\begin{bmatrix} 0 & R_I \\ L_I & 0 \end{bmatrix}$ where $R_I,L_I$ are the signed adjacency matrices introduced earlier with respect the chosen ordering. 
We can check that, indeed,
$$
\ophi_I \ophi_J + \ophi_J \ophi_I = \begin{bmatrix} R_IL_J + R_JI_I & 0 \\ 0 & L_IR_J + L_JR_I \end{bmatrix} = \mathbbm{1}_{2d},
$$
using the fact that the operators $L_*$ and $R_*$ satisfy the garden-algebra relations.

In other words, an adinkra $A$ is completely encoded by the $n$ operators $\ophi_I$ acting on $\mathbf{C}^V$, exchanging the white and black vertices in such a way that the usual garden algebra relations are satisfied. 
Once we fix a basis, i.e., a labelling of these vertices, the $L_I$ and $R_I (= L_I^t)$ matrices appear as the off-diagonal blocks of the matrix representing the $\ophi_I$. 

For each color-pair $(I,J)$, the  \emph{(fermionic) holoraumy matrix} is defined \cite{gatesLorentzCovariantHoloraumyInduced2015} by the rule:
$$\widetilde{V_{IJ}}:=\frac{1}{2i}\left(R_{I}L_{J}-R_{J}L_{I}\right) = -iR_IL_J = iR_JL_I.$$
These are $d\times d$ matrices that map fermions to fermions (hence the name). Similarly, we can define the \emph{bosonic holoraumy matrix} as
$$V_{IJ} = \frac{1}{2i}\left(L_{I}R_{J}-L_{J}R_{I}\right) = -iL_IR_J = iL_JR_I.$$

As shorthand, we write $\ophi_{I_1 I_2 \cdots I_s} = \ophi_{I_1} \cdots \ophi_{I_s}$ (so in particular, $\ophi_{IJ} = \ophi_I \ophi_J$) and refer to $\ophi_{IJ}$ as the \emph{holoraumy operator}. Because the $\ophi_I$'s are adjacency matrices, $v^{\ophi_{I_1 I_2 \cdots I_s}}$ for a vertex $v$ should capture the vertex $w$ obtained from $v$ by moving with colors $I_1, I_2, \ldots, I_s$ in order, retaining the sign (the total number of dashings modulo $2$). Specifically for the holoraumy operator, we have have
$$ \ophi_{IJ} = \ophi_I \ophi_J = \begin{bmatrix}
  0 & R_I \\
  L_I & 0 \end{bmatrix} \begin{bmatrix}
  0 & R_J \\
  L_J & 0 \end{bmatrix} = \begin{bmatrix}
  R_IL_J & 0 \\
  0 & L_IR_J \end{bmatrix}. $$

Then in terms of our $\ophi$'s, we have that $$\widetilde{V}_{IJ} = -iR_IL_J = -i(\ophi_{IJ})|_F,$$
which is just (up to a constant) the restriction of $\ophi_{IJ}$ to the subspace spanned by the fermions (the first $d$ rows and columns). Note that restricting our holoraumy operator to the bosons gives $(\ophi_{IJ})|_B = L_IR_J$, which is equal to $(-i)$ times the bosonic holoraumy matrix.

The matrix $(\ophi_{IJ})_{|_F}$ is a signed permutation matrix that sends each fermion $x$ to the fermion at the end of the path obtained by starting at $x$, travelling the edge of color $I$, and then the edge of color $J$, picking up a sign of $-1$ for every dashed edge. In this paper, these signed permutation matrices without the $i$'s are easier to work with than the holoraumy matrices $\widetilde{V}_{IJ}$, which is why we introduced the holoraumy operator. Thus, we will call them \emph{normalized holoraumy matrices}.

Finally, we denote by $\pi_I = |\ophi_I|$ the unsigned versions of the $\ophi_I$, also seen as endomorphisms on $\mathbf{C}^V$. As with the $\ophi$'s, we use $\pi_{IJ}$ to denote the composition $\pi_I\pi_J$; we have e.g. $\pi_{IJ}=|\ophi_{IJ}|$. 
Then each operator $\ophi_{IJ}$ can be written uniquely as 
$$\ophi_{IJ}=\epsilon_{IJ}\pi_{IJ},$$
where $\epsilon_{IJ}\colon\pm\mathbf{C}^F\to\{\pm 1\}$ is the map that corresponds to the sign that is picked up when travelling the edges of colors $I$ and then $J$; more precisely, $(-1)$ raised to the number of dashes the two edges contain in total. 

\begin{lem}
\label{lem:unsigned-properties}
We have the following properties for any $I \neq J$:
\begin{itemize}
\item $\pi_{IJ}^2 = \mathbbm{1}_V$. That is, $\pi_{IJ}$ is an involution.
\item $\pi_{IJ} = \pi_{JI}.$
\end{itemize}
\end{lem}
\begin{proof}
  These come immediately from Lemma~\ref{lem:operator-properties} and then literally forgetting the signs.
\end{proof}

\begin{ex}
  \label{ex:example_adinkra}
For a concrete example, see Figure~\ref{fig:example_adinkra} for a $(4,1)$ adinkra. Here,

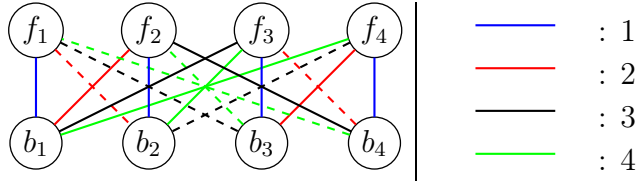
\begin{figure}
  \begin{center}
    \begin{tabular}{c|c}
\begin{tikzpicture}[scale=0.15]
\SetUpEdge[labelstyle={draw}]
\Vertex[L=$b_1$, x=-10,y=-15]{110}
\Vertex[L=$b_2$, x=0,y=-15]{101}
\Vertex[L=$b_3$, x=10,y=-15]{011}
\Vertex[L=$b_4$, x=20,y=-15]{000}

\Vertex[L=$f_1$, x=-10,y=-5]{100}
\Vertex[L=$f_2$, x=0,y=-5]{111}
\Vertex[L=$f_3$, x=10,y=-5]{001}
\Vertex[L=$f_4$, x=20,y=-5]{010}
\Edge[color=red, style=dashed](100)(101)
\Edge[color=red, style=dashed](000)(001)
\Edge[color=red](010)(011)
\Edge[color=red](110)(111)
\Edge[color=green, style=dashed](000)(100)
\Edge[color=green](001)(101)
\Edge[color=green](010)(110)
\Edge[color=green, style=dashed](011)(111)
\Edge[color=blue](000)(010)
\Edge[color=blue](001)(011)
\Edge[color=blue](100)(110)
\Edge[color=blue](101)(111)
\Edge[color=black, style=dashed](100)(011)
\Edge[color=black](000)(111)
\Edge[color=black, style=dashed](010)(101)
\Edge[color=black](110)(001)
\end{tikzpicture} & \begin{tabular}[b]{c} \begin{tikzpicture}[scale=0.15]
\SetVertexNoLabel
\GraphInit[vstyle=Empty]
\SetUpEdge[labelstyle={draw}]
\Vertex[x=0,y=-10]{A}
\Vertex[x=10,y=-10]{B}
\Edge[color=blue](A)(B)
\end{tikzpicture} : 1 \\ \begin{tikzpicture}[scale=0.15]
\SetVertexNoLabel
\GraphInit[vstyle=Empty]
\SetUpEdge[labelstyle={draw}]
\Vertex[x=0,y=-3]{A}
\Vertex[x=10,y=-3]{B}
\Edge[color=red](A)(B)
\end{tikzpicture} : 2 \\ \begin{tikzpicture}[scale=0.15]
\SetVertexNoLabel
\GraphInit[vstyle=Empty]
\SetUpEdge[labelstyle={draw}]
\Vertex[x=0,y=-3]{A}
\Vertex[x=10,y=-3]{B}
\Edge[color=black](A)(B)
\end{tikzpicture} : 3 \\ \begin{tikzpicture}[scale=0.15]
\SetVertexNoLabel
\GraphInit[vstyle=Empty]
\SetUpEdge[labelstyle={draw}]
\Vertex[x=0,y=-3]{A}
\Vertex[x=10,y=-3]{B}
\Edge[color=green](A)(B)
\end{tikzpicture} : 4 \end{tabular}
      \end{tabular}
\caption{An example $(4,1)$ adinkra and its color labels. \label{fig:example_adinkra}}
\end{center}
\end{figure}

$$\ophi_1 = \begin{bmatrix}
  0 & 0 & 0 & 0 & 1 & 0 & 0 & 0 \\
  0 & 0 & 0 & 0 & 0 & 1 & 0 & 0 \\
  0 & 0 & 0 & 0 & 0 & 0 & 1 & 0 \\
  0 & 0 & 0 & 0 & 0 & 0 & 0 & 1 \\
  1 & 0 & 0 & 0 & 0 & 0 & 0 & 0\\
  0 & 1 & 0 & 0 & 0 & 0 & 0 & 0\\
  0 & 0 & 1 & 0 & 0 & 0 & 0 & 0\\
  0 & 0 & 0 & 1 & 0 & 0 & 0 & 0 \end{bmatrix},
\ophi_2 = \begin{bmatrix}
  0 & 0 & 0 & 0 & 0 & -1 & 0 & 0 \\
  0 & 0 & 0 & 0 & 1 & 0 & 0 & 0 \\
  0 & 0 & 0 & 0 & 0 & 0 & 0 & -1 \\
  0 & 0 & 0 & 0 & 0 & 0 & 1 & 0 \\
  0 & 1 & 0 & 0 & 0 & 0 & 0 & 0\\
  -1 & 0 & 0 & 0 & 0 & 0 & 0 & 0\\
  0 & 0 & 0 & 1 & 0 & 0 & 0 & 0\\
  0 & 0 & -1 & 0 & 0 & 0 & 0 & 0 \end{bmatrix},
$$

then we have

$$\ophi_{12} = \begin{bmatrix}
  0 & 1 & 0 & 0 & 0 & 0 & 0 & 0\\
  -1 & 0 & 0 & 0 & 0 & 0 & 0 & 0\\
  0 & 0 & 0 & 1 & 0 & 0 & 0 & 0\\
  0 & 0 & -1 & 0 & 0 & 0 & 0 & 0 \\
  0 & 0 & 0 & 0 & 0 & -1 & 0 & 0 \\
  0 & 0 & 0 & 0 & 1 & 0 & 0 & 0 \\
  0 & 0 & 0 & 0 & 0 & 0 & 0 & -1 \\
  0 & 0 & 0 & 0 & 0 & 0 & 1 & 0 
 \end{bmatrix}, (\ophi_{12})_{|_F} = \begin{bmatrix}
  0 & 1 & 0 & 0 \\
  -1 & 0 & 0 & 0 \\ 
  0 & 0 & 0 & 1 \\
  0 & 0 & -1 & 0 \end{bmatrix}.
$$
Starting at fermion $f_4$, going by color $1$ (blue) then $2$ (red) moves to $b_4$ and then $f_3$, picking up a negative sign because $(f_3, b_f)$ is dashed. Thus, both the $(4,3)$ entry of $\ophi_{12}$ and that of $(\ophi_{12})_{|_F}$ (as that is just the upper-left $4$x$4$ block matrix of $\ophi_{12}$ by definition) equal $-1$, as observed.

\end{ex}


\section{$(4,1)$ adinkras}
\label{sec:4-1}

The discussion so far is about general adinkras. We now focus our attention on the $(4,1)$ adinkras, which have been key players in works such as \cite{calkinsAdinkras0branesHoloraumy2015,gatesLorentzCovariantHoloraumyInduced2015,gates:genomics,gatesjr.AdinkrasOrderedQuartets2017}. These correspond to $GR(4,4)$ garden algebras. These are adinkras (as in Example~\ref{ex:example_adinkra}) where:

\begin{itemize}
\item We have $4$ bosons $B = \{b_1, \ldots, b_4\}$ and $4$ fermions $F = \{f_1, \ldots, f_4\}$.
\item We have $4$ different colors of edges $\{1,2,3,4\}$.
\item The underlying graph is the complete bipartite graph on $B$ and $F$; that is, every pair of a boson and a fermion has a single edge.
\end{itemize}
These adinkras have many specific symmetries that we believe make the ``gadget'' quite special for $(4,1)$. To better understand such adinkras, we begin by considering their chromotopologies, which amounts to ignoring the dashing on the edges (and equivalently for the operators, ignoring the signs). 

\subsection{The Standard Chromotopology}

We introduce a specific $(4,1)$ chromotopology $C^0$ in Figure~\ref{fig:standard}, which we name the \emph{standard chromotopology}. Note that this is the underlying chromotopology of the adinkra in Example~\ref{ex:example_adinkra} as well.


\begin{figure}
\begin{center}
\begin{tabular}{c|c}
\begin{tikzpicture}[scale=0.15]
\SetUpEdge[labelstyle={draw}]
\Vertex[L=$b_1$, x=-10,y=-5]{110}
\Vertex[L=$b_2$, x=0,y=-5]{101}
\Vertex[L=$b_3$, x=10,y=-5]{011}
\Vertex[L=$b_4$, x=20,y=-5]{000}

\Vertex[L=$f_1$, x=-10,y=5]{100}
\Vertex[L=$f_2$, x=0,y=5]{111}
\Vertex[L=$f_3$, x=10,y=5]{001}
\Vertex[L=$f_4$, x=20,y=5]{010}

\Edge[color=blue](000)(010)
\Edge[color=blue](001)(011)
\Edge[color=blue](100)(110)
\Edge[color=blue](101)(111)

\Edge[color=red](100)(101)
\Edge[color=red](000)(001)
\Edge[color=red](010)(011)
\Edge[color=red](110)(111)
\Edge[color=green](000)(100)
\Edge[color=green](001)(101)
\Edge[color=green](010)(110)
\Edge[color=green](011)(111)
\Edge[color=black](000)(111)
\Edge[color=black](001)(110)
\Edge[color=black](100)(011)
\Edge[color=black](101)(010)
\end{tikzpicture} &
\begin{tabular}[b]{c}
\begin{tikzpicture}[scale=0.15]
\SetVertexNoLabel
\GraphInit[vstyle=Empty]
\SetUpEdge[labelstyle={draw}]
\Vertex[x=0,y=-3]{A}
\Vertex[x=10,y=-3]{B}
\Edge[color=blue](A)(B)
\end{tikzpicture} : 1 \\

\begin{tikzpicture}[scale=0.15]
\SetVertexNoLabel
\GraphInit[vstyle=Empty]
\SetUpEdge[labelstyle={draw}]
\Vertex[x=0,y=-3]{A}
\Vertex[x=10,y=-3]{B}
\Edge[color=red](A)(B)
\end{tikzpicture} : 2 \\

\begin{tikzpicture}[scale=0.15]
\SetVertexNoLabel
\GraphInit[vstyle=Empty]
\SetUpEdge[labelstyle={draw}]
\Vertex[x=0,y=-3]{A}
\Vertex[x=10,y=-3]{B}
\Edge[color=black](A)(B)
\end{tikzpicture} : 3 \\

\begin{tikzpicture}[scale=0.15]
\SetVertexNoLabel
\GraphInit[vstyle=Empty]
\SetUpEdge[labelstyle={draw}]
\Vertex[x=0,y=-3]{A}
\Vertex[x=10,y=-3]{B}
\Edge[color=green](A)(B)
\end{tikzpicture} : 4 

\end{tabular}
\end{tabular}
\caption{The standard chromotopology $C^0$ and the labels of colors. \label{fig:standard}}
\end{center}
\end{figure}
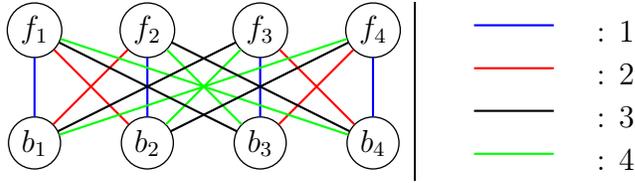

\begin{lem}
\label{lem:S_3}
Given any $(4,1)$ chromotopology $C$, there is a unique $(\rho_B, \rho_F) \in S_4 \times S_4$ such that:
\begin{itemize}
\item $\rho_F(1) = 1$, and
\item Applying $\rho_B$ and $\rho_F$ to the bosons and fermions of $C^0$ respectively gives $C$.
\end{itemize}
\end{lem}
\begin{proof}
  In $C$, take $1_F$ and place it on the upper left. It has $4$ edges coming from it to the bosons $B$, which fixes an order on the bosons. Specifically, put left-to-right on the lower level the $4$ bosons connected to $1_F$ via the colors $1,2,3,4$ respectively. Thus, the bosons are in some order $b_{(\rho(1))}, b_{(\rho(2))}, b_{(\rho(3))}, b_{(\rho(4))}$ for some $\rho \in S_4$. Set $\rho_B = \rho$. Now, put above each $b_{(\rho_B(i))}$ the fermion connected to that boson with color $1$. This means the top row is now $(f_1, f_{\rho(2)'}, f_{\rho(3)'}, f_{\rho(4)'})$ for some $\rho' \in S_4$ such that $\rho'(1) = 1$. Set $\rho_F = \rho'$. It is easy to check that the rules of the chromotopology now fix the colors of all the remaining edges (e.g. the edge between $f_{(\rho_B(1))}$, the bottom left boson, and $f_{(\rho_F(2))}$, the second fermion in the top, must be color $2$ because of the $4$-cycle property). In other words, we have drawn a copy of $C^0$ with the top $4$ fermions relabeled and the bottom right $3$ bosons relabeled, with no other choices. 
\end{proof}

Given a chromotopology $C$, we can define $(\rho_B(C), \rho_F(C))$ to be the $(\rho_B,\rho_F)$ obtained via Lemma~\ref{lem:S_3}. When the context of $C$ is clear, we omit the argument and just write $\rho_B$ or $\rho_F$.

\subsection{Chirality}

A very special property of the $(4,1)$ case is that the operators $\ophi_I$ admit another relation that is not obvious from the garden algebra relations.  Namely, we have
$$(\ophi_{1234})|_B=\chi_0(A)\mathbbm{1}_V, (\ophi_{1234})|_F=-\chi_0(A)\mathbbm{1}_V. $$
Equivalently,
$$\ophi_{1234} = \begin{bmatrix} -\chi_0(A)\mathbbm{1}_{4} & 0 \\ 0 & \chi_0(A) \mathbbm{1}_{4}  \end{bmatrix}.$$
The sign $\chi_0(A)$ is called the \emph{chirality} of the adinkra $A$ and partitions the set of $(4,1)$ adinkras into two types, ``cis'' ($\chi_0(A) = 1$) and ``trans'' ($\chi_0(A) = -1$) adinkras, modeled after the cis-trans isomerism from organic chemistry \cite{gates:genomics}. So e.g. our example adinkra in Figure~\ref{fig:example_adinkra} has $\chi_0(A) = 1$, as starting from $b_1$ and moving by the colors blue, red, black, and green in turn picks up $2$ dashes.

We now define a \emph{vertex flip} (as in e.g. \cite{douglas,douglasAutomorphismPropertiesClassification2015}) at $v$ to be the operation that takes a vertex $v \in V$ and changes the dashedness of all edges incident to $v$. It is easy to check that vertex flips do not affect chirality. This operation of vertex flipping acts transitively within each chirality-class. In general, there are $2^k$ such equivalence classes; for a general mathematical treatment and proofs relating to these equivalence classes, see the discussion of ``vertex switching classes'' in \cite{zhang:adinkras}.

\begin{lem} 
\label{lem:signed-properties}
For any $\tau \in S_4$, $\ophi_\tau = (-1)^{s(\tau)} \begin{bmatrix} -\chi_0(A)\mathbbm{1}_{4} & 0 \\ 0 & \chi_0(A) \mathbbm{1}_{4}  \end{bmatrix},$ where $s(\tau)$ is the sign of the permutation $\tau$.
\end{lem}
\begin{proof}
  For $\tau = 1234$, the identity in $S_4$, this is true by definition of chirality. Then, we can get to any other permutation transitively by adjacent transpositions of $2$ indices. Each transposition changes the sign of the permutation by $-1$, but also changes $\ophi_{\tau}$ by $-1$ as we know from Lemma~\ref{lem:operator-properties} that $\ophi_{I}\ophi_J = -\ophi_J \ophi_I$. Thus, the property is true for all $\tau$.
  \end{proof}

We say that $\{I, J\}$ and $\{K, L\}$ are \emph{complementary} if $\{I,J,K,L\} = \{1,2,3,4\}$. We immediately obtain:
\begin{cor}
  \label{cor:signed-properties}
  The following hold for an $(4,1)$ adinkra $A$ and any $I \neq J$:
  \begin{enumerate}
  \item If $\{K,L\}$ is complementary to $\{I,J\}$, $\ophi_{IJ} \ophi_{KL} = \ophi_{KL}\ophi_{IJ}$.
  \item $\ophi_{1234} = (-\ophi_{1324}) = \ophi_{1423} = \begin{bmatrix} -\chi_0(A)\mathbbm{1}_{4} & 0 \\ 0 & \chi_0(A) \mathbbm{1}_{4}  \end{bmatrix}.$
  \item $\ophi_{12} = \begin{bmatrix} \chi_0(A)(\ophi_{34})_{|_F} & 0 \\ 0 & -\chi_0(A) (\ophi_{34})_{|_B}\end{bmatrix}.$
  \item $\ophi_{13} = \begin{bmatrix} -\chi_0(A)(\ophi_{24})_{|_F} & 0 \\ 0 & \chi_0(A) (\ophi_{24})_{|_B}\end{bmatrix}.$
  \item $\ophi_{14} = \begin{bmatrix} \chi_0(A)(\ophi_{23})_{|_F} & 0 \\ 0 & -\chi_0(A) (\ophi_{23})_{|_B}\end{bmatrix}.$
  \end{enumerate}
\end{cor}
\begin{proof}   The first item comes from the fact that $\ophi_{IJ}\ophi_{KL} = \ophi_I \ophi_J \ophi_K \ophi_L$ and it takes $6$,  (an even number) ``swaps'' of the form $\ophi_X \ophi_Y = -\ophi_Y \ophi_X$ to get from one expression to the other. The second item is immediate from Lemma~\ref{lem:signed-properties}. To get the third item, start with $\ophi_{1234} = \begin{bmatrix} -\chi_0(A)\mathbbm{1}_{4} & 0 \\ 0 & \chi_0(A) \mathbbm{1}_{4}  \end{bmatrix}$ and multiply both sides on the right by $\phi_{43}$. The remaining items are analogous to the third item.
\end{proof}

\begin{lem} 
\label{lem:unsigned-properties-4-1}
We have the following properties for any $I \neq J$, where $\{K,L\}$ is complementary to $\{I,J\}$:
\begin{enumerate}
\item $\pi_{IJ} = \pi_{KL}$. 
\item $\pi_{IJ} \pi_{IK} = \pi_{JK} = \pi_{IL}$.
\end{enumerate}
\end{lem}
\begin{proof}
  The first item is immediate from using Corollary~\ref{cor:signed-properties} and then forgetting the signs. For example, the third item of Corollary~\ref{cor:signed-properties} becomes $$\pi_{12} = |\ophi_{12}| = \begin{bmatrix} (\pi_{34})_{|_F} & 0 \\ 0 &  (\pi_{34})_{|_B}\end{bmatrix} = \pi_{34}.$$
  For the second item, it suffices to check by hand on the standard chromotopology $C^0$, as relabeling the bosons and fermions do not affect these relationships. We omit the technical details.
\end{proof}

We know from Lemma~\ref{lem:unsigned-properties} that $\pi_{IJ}^2 = \mathbbm{1}_V$. Combining this with Lemma~\ref{lem:unsigned-properties-4-1}, we see that we have just checked the group multiplication table for $K_4$, the \emph{Klein-$4$ group}! This gives us the following result:
\begin{cor}
  \label{cor:K_4}
  We have the following properties:
  \begin{itemize}
  \item The identity matrix $\mathbbm{1}_8$ and the three elements $\pi_{12} = \pi_{34}$, $\pi_{13} = \pi_{24}$, $\pi_{14} = \pi_{23}$ form a group isomorphic to $K_4$.
    \item Restricting to the fermions, $\{\mathbbm{1}_4, (\pi_{12})|_F, (\pi_{13})|_{F}, (\pi_{14})|_F\}$ form a group isomorphic to $K_4$. The same holds when we restrict to the bosons.
  \end{itemize}
\end{cor}

We denote by $\K:=K_4 - \{1\}$ the set of the $3$ non-identity elements of $K_4$. We now know that given an $(4,1)$ adinkra $A$, the $6$ normalized holoraumy matrices $(\ophi_{IJ})_{|_F}$ (for $I < J$) come in $3$ pairs, one for each element of $\K$; this is because when we forget the signs, the $(\ophi_{IJ})_{|_F}$'s become $(\pi_{IJ})|_{|_F}$'s, which have the $K_4$ structure described in Corollary~\ref{cor:K_4}.

\begin{remark}
\label{rem:count}
Let us count, in our terms, the total number of possible adinkra representations. 
Ignoring the odd dashing for now, there are exactly $|S_4|$ choices of bijections from the bosons $B$ to the fermions $F$, which means that there are $|S_4|$ choices for the operator $\pi_1$. 
Once the operator $\pi_1$ is chosen, the operators $\pi_1\pi_2,\pi_1\pi_3,\pi_1\pi_4$ restricted to $B$ must then correspond to the elements of $\K$, as discussed above. 
Once a bijection of $\pi_1\pi_2,\pi_1\pi_3,\pi_1\pi_4$ with $\K$ is chosen, the permutations $\pi_i$ for $i=2,3,4$ are completely determined. 
Putting this all together, there are $|S_3|\cdot|S_4|$ ways of choosing the permutation operators $\pi_i$. 
Finally, there are $2^8$ odd dashings one can place on the underlying graph (\cite{gatesjr.AdinkrasOrderedQuartets2017} for the $(4,1)$ case and \cite{zhang:adinkras} for general theory), from which we obtain a total of 
\[2^8\cdot 24\cdot 6=36864\]
distinct adinkras. 

While there have been several other versions of this count, this perspective allows us to see that we may \textbf{fix} a choice of ordered bases for $V = B \cup F$ and cycle through all of the adinkras by varying the choices of operators $\pi_I$.
\end{remark}

In particular, $\rho_B$ (from Lemma~\ref{lem:S_3}) is some permutation in $S_4$ that permutes $\{2,3,4\}$ and keeps $1$ fixed. There are $6$ such permutations (which we write in cycle form):
\begin{itemize}
\item the identity permutation $()$ of cycle type $(1)(1)(1)$;
\item the swaps $(23)$, $(24)$, and $(34)$ of cycle type $(2)(1)$;
\item the cycles $(234)$ and $(432)$ of cycle type $(3)$.
\end{itemize}
In addition to the $6$ ways of choosing $\rho_B$, there are $|S_4| = 24$ ways to choose $\rho_F$, and each pair of choices gives a different chromotopology. In light of Remark~\ref{rem:count}, this accounts for all $24*6=144$ $(4,1)$ chromotopologies.

The upshot of Lemma~\ref{lem:S_3} is that it allows us to compute $\pi_{IJ}$ for an arbitrary chromotopology $C$:
\begin{lem}
\label{lem:standard}
Suppose an adinkra $A$ has chromotopology $C(A)$. Then 
\[
\pi_{IJ} = (\rho_F)(IJ)(KL)(\rho_F^{-1}).
\]
\end{lem}
\begin{proof}
We chose the colors of $C^0$ so that in $C^0$, $\pi_{IJ} = \pi_{KL} = (IJ)(KL)$, where $I,J$ and $K,L$ are complementary; the last expression is for expressing a permutation in cycle form. By Lemma~\ref{lem:S_3}, $C$ is obtained from $C^0$ by $\rho_F(C)$ applied to the non-$1_F$ bosons of $C^0$ and some $\rho_B$ applied to the bosons of $C^0$. Therefore, following colors $J$ and then $I$ from any vertex $i_F$ means first applying $\rho^{-1}(i)$ to find the corresponding vertex label in $C^0$, following  $\pi_{IJ}$ in $C^0$ (which gives $(IJ)(KL)$), and then finding the correct label in $C(A)$ by applying $\rho$.
\end{proof}

This Lemma also makes clear which pair of operators corresponds to each element of $\K$. For example, if $\rho_F(C(A)) = (234)$, then we know
\[
\pi_{12} = (234)(12)(34)(432) = (13)(24),
\]
so we know $\widetilde{V}_{12}$ and $\widetilde{V}_{3,4}$ correspond to $(13)(24)$, as their underlying unsigned permutation matrices do.

\subsection{Degrees of Freedom}

Given a $(4,1)$ adinkra $A$ with garden algebra generators $\{L_i\}$, consider the $4 \times 4$ matrix 
\[
E = L_1 + L_2 + L_3 + L_4,
\]
with rows indexed by the bosons and columns indexed by fermions. This gives us a matrix wherein every element is $1$ or $-1$, because the underlying graph of $A$ is the complete bipartite graph $K_{4,4}$. We call this the \emph{edge matrix} of $A$ because each entry corresponds to one of the edges.

\begin{lem}
\label{lem:4-1-dashings-freedom}
Given an $(4,1)$ chromotopology $C$, a chirality $\chi_0$, a color $c$, and a vertex $v$, there are $2^{7}$ choices to (freely) select the dashings of the $4$ $c$-colored edges and the $3$ other edges incident to $v$. Each such choice can be completed in a unique way to an odd dashing of $C$ (thus creating a valise $(4,1)$ adinkra).
\end{lem}
\begin{proof}
We prove this case when all the edges of color $c=1$ are solid and $v = b_1$. All the other cases follow from symmetry. Consider the edge matrix $E$ of $A$. Without loss of generality, assume $C = C^0$, so the edges of color $1$ are exactly the $E_{i,i}$. We freely select a sign for each. Let these choices be $e_i = E_{i,i}$ for all $i \in \{1,2,3,4\}$.

We know that for all $i \neq j$, following $E_{i,i}, E_{i,j}, E_{j,i}, E_{i,j}$ gives a $2$-colored $4$-cycle, so the product of the $4$ signs must be $(-1)$; however, we already selected values for $E_{i,i} = e_i$ and $E_{j,j} = e_j$, so we get $E_{i,j} = - e_i e_j E_{j,i}$ for all $i \neq j$.

Now, we also freely parametrize $x = E_{1,2}, y = E_{1,3}, z = E_{1,4}$. By now, we have made choices of the $7$ dashings described by the Lemma. It remains to show that each of these $2^7$ choices uniquely fixes all the remaining signs.

Recall that having chirality $\chi_0(A)$ means that starting with any boson, following colors $1$,$2$,$3$,$4$ in order gives total sign $\chi_0(A)$. Let $s = \chi_0(A)$. One such path is 
\[
b_1 \rightarrow_1 f_1 \rightarrow_2 b_2 \rightarrow_3 f_4 \rightarrow_4 b_1,
\] 
where we use subscripts of the arrow to denote color. Multiplying the dashings, we obtain $(e_1)(-e_1e_2x)(E_{24})(z) = s$, so $E_{2,4} = -se_2xz$.

Doing the same for the path 
\[
b_2 \rightarrow_1 f_2 \rightarrow_2 b_1 \rightarrow_3 f_3 \rightarrow_4 b_2
\] 
gives the product $(e_2)(x)(y)(E_{2,3}) = s$, so $E_{2,3} = se_2xy$.

Finally, following 
\[
b_3 \rightarrow_1 f_3 \rightarrow_2 b_4 \rightarrow_3 f_2 \rightarrow_4 b_3
\]
gives the product $(e_3)(-e_3e_4E_{3,4})(-e_2e_4E_{2,4})(-e_2e_3E_{2,3}) = s$, so $E_{3,4} = -se_3E_{2,3}E_{2,4} = -se_3yz$.

We can check that these values $E_{i,j}$ form a legitimate dashing by looking at all the remaining $4$-cycles. 
\end{proof}

As one last sanity check, the above proof claims there are $2*2^7 = 2^8$ (the first term for chirality and the second term for the $7$ edges) dashings for a particular chromotopology, which is what we expect from the general dashing counting formula $2^{2^{n-k}-k+1} = 2^{2^3-1+1}$ found in \cite{zhang:adinkras}.

\section{Our Two Gadgets}
\label{sec:gadgets}

Fix a pair of $(4,1)$ adinkras $A$ and $A'$, and use $\ophi_I, \ophi_{IJ}, \widetilde{V}_{IJ}$ to denote the matrices coming from $A$, and $\ophi_I', \ophi_{IJ}', \widetilde{V}'_{IJ}$ for the matrices coming from $A'$. 
Then Gates \cite{gatesLorentzCovariantHoloraumyInduced2015}  defines the \emph{gadget} as 
\[\mathcal{OG}(A,A'):=\frac{1}{48}\sum_{I \neq J}\tr(\widetilde{V}_{IJ} \widetilde{V}'_{IJ}).\]
We can rewrite this as
\[ -\frac{1}{48}\sum_{I \neq J}\tr_{|_F}\left(\ophi_{IJ} \ophi_{IJ}'\right) = -\frac{1}{24}\sum_{I < J} \tr_{|_F}(\ophi_{IJ}\ophi_{IJ}'), \]
where in the last line we use the following symmetry: for any term where $J < I$, because $\phi_{IJ} = -\phi_{JI}$ we know that $\ophi_{JI}\ophi_{JI}' = \ophi_{IJ}\ophi_{IJ}'$. Instead of summing over all $12$ cases, we can just sum over $6$ the cases where $I < J$. 

\begin{remark}
  \label{rem:naturality}
  One issue with the product $\widetilde{V}_{IJ}\widetilde{V}'_{IJ}$ is that does not admit an upfront interpretation as a composition of operators: the holoraumy matrices are $4\times 4$ matrices representing operators with respect to \textbf{different} spaces. Our discussion in Section~\ref{sec:4-1} identifies these bases. For each vertex set $V$, picking a fixed set of vertex names $\{b_1, \ldots, b_4, f_1, \ldots, f_4\}$ induces an isomorphism $\mathbf{C}^V\cong\mathbf{C}^8$. Each $\ophi_I$ induces an operator on $\mathbf{C}^V$ (being identified with the $\mathbf{C}$-span of the vertices in each adinkra), which we also denote by $\ophi_I$ to avoid more cumbersome notation. Given adinkras $A$ and $A'$, the composition $\ophi_I\ophi'_J$ is now the composition of operators on the \textbf{same} space $\mathbf{C}^V$. Similarly, the normalized holoraumy matrices are now endomorphisms on the same $4$-dimensional subspace $\mathbf{C}^F$ (corresponding to fermions) of said $\mathbf{C}^V$. 
  \end{remark}


 
Interpreting the normalized holoraumy matrices as operators on the same vector space, as in Remark~\ref{rem:naturality}, motivates us to introduce a new gadget that appears naturally from linear algebra. Given two endomorphisms $X$ and $Y$ on the same vector space, there is a natural non-degenerate symmetric bilinear form on endomorphisms:
\[\gen{X,Y}:=\tr(X\cdot Y).\]
Guided by this, to each adinkra $A$ we associate the operator
\[X_A:= \sum_{I < J} (\ophi_{IJ})|_F\]
on $\mathbf{C}^F$ and define the \emph{symmetrized gadget} as follows:
\[\mathcal{SG}(A,A'):=\gen{X_A,X_{A'}}=\sum_{I_1<J_1,I_2<J_2}\tr_{|_F}(\ophi_{I_1J_1}\ophi'_{I_2J_2}).\]

As $\mathcal{OG}(A,A')$ equals, up to a constant, $\sum_{I<J}\tr_{|_F}(\ophi_{IJ}\ophi'_{IJ})$, the symmetrized gadget can be thought of as being obtained from (a constant multiple of) the original gadget by adding additional terms corresponding the other products coming from the very natural object $\tr(X_A \cdot X_A')$. 

\begin{remark}
We can also think of the operator $X_A$ with signed adjacency matrices. Specifically, consider the graph with vertex set equal to $F$, the fermions. For each pair of vertices $x$ and $y$, sum up the length-$2$ paths going from $x$ to $y$ via colors $I$ and $J$ in order in the adinkra for each $I<J$, allowing negative signs for each dash. Then, $X_A$ is the signed weighted adjacency matrix for this graph.	It should be noted that since this graph has multiple edges (in fact, each pair of vertices is joined by exactly $2$ two edges) and since the two edges may have different signs, it is possible for two vertices to be adjacent in this graph even though the corresponding entry in the matrix for $X_A$ equals $0$.
\end{remark}

It turns out that the two objects enjoy similar properties. Recall from Section~\ref{sec:operators} that we can think of an operator $\ophi_{IJ}$ as a product $\epsilon_{IJ}\pi_{IJ}$, where the first term is the signs and the second term is a permutation matrix. Formally,
\begin{eqnarray*}
	\ophi_{I_1J_1}\ophi'_{I_2,J_2}&=&\epsilon_{I_1J_1}\pi_{I_1J_1}\epsilon'_{I_2J_2}\pi'_{I_2J_2}\\
		&=&\epsilon_{I_1J_1}\pi_{I_1J_1}\epsilon'_{I_2J_2}{\pi'}_{I_1J_1}^{-1}\pi_{I_1J_1}\pi_{I_2J_2}\\
		&=&\nu_{I_1J_1,I_2J_2}\pi_{I_1J_1}\pi'_{I_2J_2},
\end{eqnarray*}
where 
$$\nu_{I_1J_1,I_2J_2}:=\epsilon_{I_1J_1}\pi_{I_1J_1}\epsilon'_{I_2J_2}{\pi'}_{I_1J_1}^{-1}.$$
Note that $\nu_{I_1J_1,I_2J_2}\colon\mathbf{C}^F\to\{\pm 1\}$ and can be thought of as a tally of the signs picked up as we travel along the edges of color $I_1$ and then $J_1$ in $A$, followed by edges of color $I_2$ and $J_2$ in $A'$ (it is useful to think of $A$ and $A'$ ``glued'' to each other along the fermions $F$). This allows us to write the gadgets as follows: 
\begin{align*}
  \mathcal{G}(A,A') & = \sum_{I,J}\tr_{|_F}(\nu_{IJ,IJ}\pi_{IJ}\pi'_{IJ}) \\
  \mathcal{SG}(A,A')& = \sum_{I_1<J_1,I_2<J_2}\tr_{|_F}(\nu_{I_1J_1,I_2J_2}\pi_{I_1J_1}\pi'_{I_2J_2}).
\end{align*}
Since the $\nu$-operators are diagonal, our strategy will be to study the fixed-points of the permutation operators $\pi_{I_1J_1}\pi'_{I_2J_2}$. In the following $2$ sections, we execute this strategy.

\section{Values of the Gadget}
\label{sec:yan}

Table~\ref{table:computation} lists the frequencies found in \cite{gatesjr.AdinkrasOrderedQuartets2017} via computer for the gadget values for all the $36864^2$ possible ordered pairs of adinkras. In this section, we confirm these results ``by hand'' via algebra and also derive structural properties that allow us to understand, completely, how these numbers arose.

\begin{table}[ht]
\caption{Computations from \cite{gatesjr.AdinkrasOrderedQuartets2017}.}
\begin{center}
\begin{tabular}{|c|c|}
\hline
  value of $\mathcal{OG}(A,A')$ & number of pairs $(A,A')$ \\
\hline
$1$ & $14,155,776$ \\   
$1/3$ & $84,934,656$ \\  
$-1/3$ & $127,401,984$ \\ 
$0$ & $1,132,462,080$ \\
\hline
\end{tabular}
\end{center}
\label{table:computation}
\end{table}

Our approach starts with the following lemma:

\begin{lem}
\label{lem:cover}
Suppose adinkras $A$ and $A'$ have corresponding operators $\ophi_*$ and $\ophi'_*$. We have, for any distinct $I < J$ in $\{1,2,3,4\}$, 
\[
\ophi_{IJ}\ophi'_{IJ} = \chi_0(A)\chi_0(A') \ophi_{KL}\ophi'_{KL},
\]
where $K < L$ are complementary to $\{I, J\}$.

\end{lem}
\begin{proof}
We do the proof for $I = 1$ and $J=2$; the other situations are similar.

We know that $$\ophi_1\ophi_2\ophi_3\ophi_4 = \begin{bmatrix} -\chi_0(A)\mathbbm{1}_{4} & 0 \\ 0 & \chi_0(A) \mathbbm{1}_{4}  \end{bmatrix}, \ophi_3\ophi_4 = \begin{bmatrix}(\ophi_{34})|_F & 0 \\ 0 & (\ophi_{34})|_B\end{bmatrix}, (\phi_3 \phi_4)^2 = - \mathbbm{1}_8.$$ Thus, $$\ophi_1\ophi_2 = -\begin{bmatrix} -\chi_0(A)\mathbbm{1}_{4} & 0 \\ 0 & \chi_0(A) \mathbbm{1}_{4}  \end{bmatrix} \begin{bmatrix}(\ophi_{34})|_F & 0 \\ 0 & (\ophi_{34})|_B\end{bmatrix} = \begin{bmatrix} \chi_0(A)(\ophi_{34})|_F & 0 \\ 0 & -\chi_0(A) (\ophi_{34})|_B \end{bmatrix}.$$ By same reasoning on $A'$, we obtain
$$\ophi'_1\ophi'_2 = \begin{bmatrix} \chi_0(A')(\ophi'_{34})|_F & 0 \\ 0 & -\chi_0(A') (\ophi'_{34})|_B \end{bmatrix}.$$
Multiplying, we obtain
$$\ophi_{12}\ophi'_{12} = \chi_0(A)\chi_0(A') \ophi_{34}\ophi'_{34},$$
as desired. For other cases, the analysis is basically the same.
\end{proof}

We now manipulate $\mathcal{OG}$. Recall that 
\[ \mathcal{OG}(A,A') = -\frac{1}{48}\sum_{I \neq J} \tr_{|_F}((\ophi_{IJ}\ophi_{IJ}') =  -\frac{1}{24}\sum_{I < J} \tr_{|_F}(\ophi_{IJ}\ophi_{IJ}').\]
Let $g_{IJ} = (\ophi_{IJ}\ophi'_{IJ})_{|_F}$. Then,
\begin{cor}
\label{cor:twofold}
For any $(4,1)$ adinkras $A$ and $A'$, we have
\[
\mathcal{OG}(A,A')= -\frac{1}{24} [\tr(g_{12}) + \tr(g_{13}) + \tr(g_{14})](1 + \chi_0(A)\chi_0(A')).
\]
In particular, if $A$ and $A'$ have different chirality, then $\mathcal{OG}(A,A') = 0$.
\end{cor}
\begin{proof}
There are $6$ cases where $I < J$. Because we have e.g. $\ophi_{34}\ophi_{34}' = \chi_0(A)\chi_0(A')\ophi_{12}\ophi_{12}'$ by Lemma~\ref{lem:cover}, the $6$ cases come in $3$ such pairs. Collecting coefficients in each pair gives the result.
\end{proof}

\begin{remark}
Lemma~\ref{lem:cover} and Corollary~\ref{cor:twofold} show that our gadget $\mathcal{G}(A,A')$ in some sense a ``2-cover'' of the Klein-$4$ group; this is a very specific behavior for $(4,1)$. This suggests that we may want a different definition of the \emph{gadget} for higher $n$ and $k$, or at least expect that the ``obvious'' generalizations are not as nice.
\end{remark}
Our strategy, then, is to compute 
\[
g(A,A'): = \tr(g_{12}) + \tr(g_{13}) + \tr(g_{14})
\] 
when $A$ and $A'$ have matching chirality. In that situation, $(1 + \chi_0(A)\chi_0(A')) = 2$, and $\mathcal{OG}(A,A') = -\frac{g(A,A')}{12}.$ Otherwise, $\mathcal{OG}(A,A')$ must be $0$ as $(1 + \chi_0(A)\chi_0(A')) = 0$. We now aim to compute $\tr(g_{IJ})$. Let $\rho = \rho_F(C(A))$ and $\rho' = \rho_F(C(A'))$.

Recall that we can decompose $g_{IJ} = (\ophi_{IJ}\ophi'_{IJ})_{|_F}$ into $(\nu_{IJ,IJ}\pi_{IJ}\pi'_{IJ})_{|_F}$. We use $\nu_{IJ}$ as shorthand for $\nu_{IJ, IJ}$ since this is the only $\nu$ type that comes up in this section. Then, by Lemma~\ref{lem:standard},
\begin{align*}
  \tr(g_{IJ}) & = \tr_{|_F}(\nu_{IJ}\pi_{IJ}\pi'_{IJ}) \\
& = \tr_{|_F}\bigl( \nu_{IJ}  (\rho)(IJ)(KL) (\rho^{-1}) (\rho') (IJ)(KL) (\rho'^{-1}) \bigr).
\end{align*}



Note that as $g_{IJ}$ is a signed permutation matrix, it can only have nonzero trace if $|g_{IJ}|$ has nonzero trace. So we temporarily forget about the sign and compute:
\begin{align*}
\tr(|g_{IJ}|) & = \tr_{|_F}\bigl( (\rho)(IJ)(KL) (\rho^{-1}) (\rho') (IJ)(KL) (\rho'^{-1}) \bigr); \\
& = \tr_{|_F}\bigl( (\rho'^{-1} \rho)(IJ)(KL) (\rho'^{-1} \rho)^{-1} (IJ)(KL) \bigr) \\
& = \tr_{|_F}\bigl( \pi  (IJ)(KL) \bigr),
\end{align*}
where $\pi$ is a permutation with cycle type $(2)(2)$ (that is, an element of $\K$) because conjugation preserves cycle type. We also used that trace is invariant under cyclic permutation. The product $(\pi)(IJ)(KL)$, being an element of $K_4 = \K \cup \{()\}$, has cycle type either $(2)(2)$, in which case it contributes no trace, or $(1)(1)(1)(1)$, in which case it is the identity. Thus, if we have nontrivial trace, we need to be in the latter case:
\begin{equation}
\label{eqn:identity}
(\rho'^{-1} \rho)(IJ)(KL) (\rho'^{-1}\rho)^{-1} = (IJ)(KL).
\end{equation}

We now fix $I=1$. There are $6$ choices for $(\rho'^{-1}\rho)$ (recall that the condition is that we have a permutation in $S_4$ that fixes $b_1$, creating a $S_3$ subgroup inside $S_4$) with $3$ possible cycle types $(3), (2)(1)$, or $(1)(1)(1)$ when seen as permutations on $\{b_2,b_3,b_4\}$ (they would have cycle types $(3)(1), (2)(1)(1)$, or $(1)(1)(1)(1)$ as permutations on $B = \{b_1, b_2, b_3, b_4\}$). Let $M = 36864$ be the number of $(4,1)$ adinkras. We now split into cases based on the cycle  and count in terms of $M$.

\begin{remark}
\label{rem:simul}
If we relabel the fermions simultaneously for both adinkras, that corresponds to an additional conjugation of both $(\rho)(IJ)(KL) (\rho^{-1})$ and $(\rho') (IJ)(KL) (\rho'^{-1})$ by some permutation $\rho''$. Since relabeling the fermions does not change the edge dashings and the trace operator is symmetric over the fermions, we know that this operation preserves $\tr(g_{IJ})$ for all $I$ and $J$. This justifies e.g. replacing $\rho' = ()$ and replacing $\rho$ by $\rho'^{-1} \rho$.  
\end{remark}






\subsection{The $(3)$ case}

It is easy to check that no $\rho'^{-1} \rho$ which acts with cycle type $(3)$ on $\{f_2,f_3,f_4\}$ satisfies Equation~\ref{eqn:identity}, so the trace is $0$ in this case. $2$ out of the $6$ permutations in $S_3$ have this cycle type, so  $M^2/3$ pairs $(A,A')$ are in this case and all satisfy $g(A,A') = 0$ (and thus, $\mathcal{OG}(A,A') = 0$ as well).



\subsection{The $(2)(1)$ case}

After fixing $A'$, half of the possible $C(A)$ fall into the $(2)(1)$ case (as there are $3$ permutations with this cycle type out of the $6$ in $S_3$). Out of these $M^2/2$ pairs, recall that we need $A$ and $A'$ to have matching chirality, which happens with half of the choices of $A$. This means $M^2/4$ of the pairs $(A,A')$ have $(1 + \chi_0(A)\chi_0(A')) = 2 \neq 0$ and have the desired cycle type for $\rho'^{-1}\rho$. Suppose we have one of these pairs.

Without loss of generality, we can assume $\rho' = ()$ and $C(A) = C^0$  by Remark~\ref{rem:simul}. To get cycle type $(2)(1)$ on $\{2,3,4\}$, we can also assume $\rho = (3 4)$, so the chromotopologies look as they are in Figure~\ref{fig:(2)(1)}.

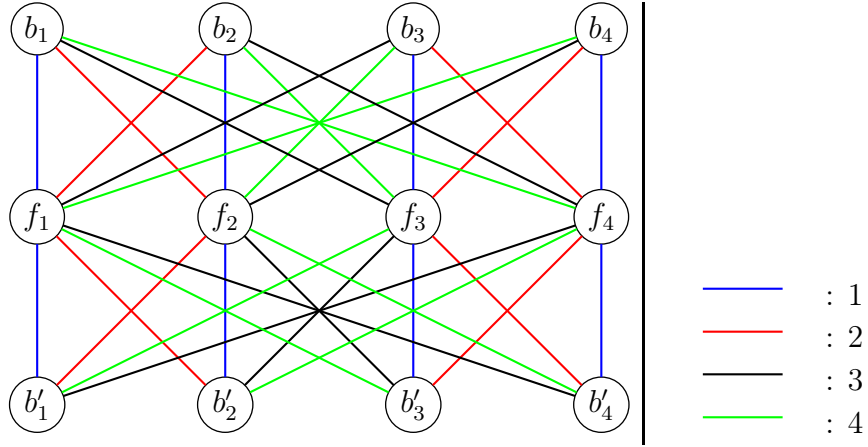
\begin{figure}
\begin{center}
\begin{tabular}{c|c}
\begin{tikzpicture}[scale=0.25]
\SetUpEdge[labelstyle={draw}]
\Vertex[L=$f_1$, x=-10,y=-5]{110}
\Vertex[L=$f_2$, x=0,y=-5]{101}
\Vertex[L=$f_3$, x=10,y=-5]{011}
\Vertex[L=$f_4$, x=20,y=-5]{000}

\Vertex[L=$b_1$, x=-10,y=5]{100}
\Vertex[L=$b_2$, x=0,y=5]{111}
\Vertex[L=$b_3$, x=10,y=5]{001}
\Vertex[L=$b_4$, x=20,y=5]{010}

\Vertex[L=$b_1'$, x=-10,y=-15]{100'}
\Vertex[L=$b_2'$, x=0,y=-15]{111'}
\Vertex[L=$b_3'$, x=10,y=-15]{001'}
\Vertex[L=$b_4'$, x=20,y=-15]{010'}

\Edge[color=blue](000)(010)
\Edge[color=blue](001)(011)
\Edge[color=blue](100)(110)
\Edge[color=blue](101)(111)
\Edge[color=red](100)(101)
\Edge[color=red](000)(001)
\Edge[color=red](010)(011)
\Edge[color=red](110)(111)
\Edge[color=green](000)(100)
\Edge[color=green](001)(101)
\Edge[color=green](010)(110)
\Edge[color=green](011)(111)
\Edge[color=black](000)(111)
\Edge[color=black](001)(110)
\Edge[color=black](100)(011)
\Edge[color=black](101)(010)

\Edge[color= blue](010')(000)
\Edge[color= blue](001')(011)
\Edge[color= blue](100')(110)
\Edge[color= blue](111')(101)
\Edge[color=  red](100')(101)
\Edge[color=  red](001')(000)
\Edge[color=  red](010')(011)
\Edge[color=  red](111')(110)
\Edge[color=black](100')(000)
\Edge[color=black](001')(101)
\Edge[color=black](010')(110)
\Edge[color=black](111')(011)
\Edge[color=green](111')(000)
\Edge[color=green](001')(110)
\Edge[color=green](100')(011)
\Edge[color=green](010')(101)

\end{tikzpicture} &
\begin{tabular}[b]{c}
\begin{tikzpicture}[scale=0.15]
\SetVertexNoLabel
\GraphInit[vstyle=Empty]
\SetUpEdge[labelstyle={draw}]
\Vertex[x=0,y=-3]{A}
\Vertex[x=10,y=-3]{B}
\Edge[color=blue](A)(B)
\end{tikzpicture} : 1 \\

\begin{tikzpicture}[scale=0.15]
\SetVertexNoLabel
\GraphInit[vstyle=Empty]
\SetUpEdge[labelstyle={draw}]
\Vertex[x=0,y=-3]{A}
\Vertex[x=10,y=-3]{B}
\Edge[color=red](A)(B)
\end{tikzpicture} : 2 \\

\begin{tikzpicture}[scale=0.15]
\SetVertexNoLabel
\GraphInit[vstyle=Empty]
\SetUpEdge[labelstyle={draw}]
\Vertex[x=0,y=-3]{A}
\Vertex[x=10,y=-3]{B}
\Edge[color=black](A)(B)
\end{tikzpicture} : 3 \\

\begin{tikzpicture}[scale=0.15]
\SetVertexNoLabel
\GraphInit[vstyle=Empty]
\SetUpEdge[labelstyle={draw}]
\Vertex[x=0,y=-3]{A}
\Vertex[x=10,y=-3]{B}
\Edge[color=green](A)(B)
\end{tikzpicture} : 4 

\end{tabular}
\end{tabular}
\caption{Left: $C(A) \cong C^0$ on top and $C(A')$ on bottom, sharing the fermions. The $b_i$ are the bosons for $A$ and the $b'_i$ are bosons for $A'$. Since $\rho=(34)$, the edges of colors $3$ and $4$ are switched between the two adinkras. \label{fig:(2)(1)}}
\end{center}
\end{figure}

The only term to provide nonzero trace to Equation~\ref{eqn:identity} is $g_{12}$, because $(12)(34)$ is the only $(IJ)(KL)$ to be invariant under conjugation by $\rho = (34)$, so
\begin{align*}
g(A,A') = \tr(g_{12}) & = \tr_{|_F}\bigl( \nu_{12} (\rho)(12)(34) (\rho^{-1}) (\rho') (12)(34) (\rho'^{-1}) \bigr) = \tr_{|_F}\bigl( \nu_{12}).
\end{align*}

By performing vertex flips on the bosons in both adinkras, we can assume all the color $1$ edges are solid (and preserve chiralities). By the proof of Lemma~\ref{lem:4-1-dashings-freedom}, we may parametrize $x = E_{1,2}, y = E_{1,3}, z = E_{1,4}$ and obtain $E_{2,3} = sxy, E_{2,4} = -sxz, E_{3,4} = -syz$ in $A$, where $s = \chi_0(A)$. Similarly, we can parametrize $x', y', z'$ for $A'$ (which we know has the same chirality) and obtain values $E'_{2,3} = sx'y', E'_{2,4} = -sx'z', E'_{3,4} = -sy'z'$. The only difference is that e.g. $E_{2,4}$ has color $3$ in $A$ but $E'_{2,4}$ has color $4$ in $A'$. We can now compute:
\begin{align*}
\tr(g_{12}) & =  E_{1,1}E_{2,1}E_{2,2}'E_{1,2}' + E_{2,2}E_{1,2}E_{1,1}'E_{2,1}'  + E_{3,3}E_{4,3}E_{4,4}'E_{3,4}' + E_{4,4}E_{3,4}E_{3,3}'E_{4,3}' \\
& =  E_{2,1}E_{1,2}' + E_{1,2}E_{2,1}' + E_{4,3}E_{3,4}' + E_{3,4}E_{4,3}' \\
& =  -2 xx' - 2 yy'zz'. 
\end{align*}

By the symmetry proven in Lemma~\ref{lem:4-1-dashings-freedom}, the parameters $\{x, x', y, y', z, z'\}$ are each equally distributed among the $2$ choices of $\pm 1$ in our set of pairs $(A,A')$. This means both $(-2xx')$ and $(-2yy'zz')$ are independently uniformly distributed in $\{\pm 2\}$. Thus, $1/4$ of the time we get $\tr(g_{12}) = -4$, $1/4$ of the time we get $4$, and $1/2$ of the time we get $0$. 

Recalling that $M^2/4$ pairs satisfy our criteria, we know that

\begin{itemize}
\item $(M^2/2 - M^2/16 - M^2/16) = 3M^2/8$ pairs have gadget value $0$ from either having $g(A,A') = 0$ or having different chiralities;
  \item $(1/4)(M^2/4) = M^2/16$ pairs have the same chirality and also satisfy $g(A, A') = 4$, so $\mathcal{OG}(A,A') = -\frac{4}{12} = -1/3$;
\item finally, $M^2/16$ pairs with matching chirality satisfy $g(A,A') = -4$, so they have $\mathcal{OG}(A,A') = 1/3$.

  \end{itemize}
  


\subsection{The $(1)(1)(1)$ Case}

Recall that this is the case where $(\rho'^{-1}\rho) = ()$; equivalently, $\rho = \rho'$. Seen as a permutation on $\{2,3,4\}$, such a permutation has cycle type $(1)(1)(1)$.  After fixing $C(A')$, $1/6$ of the possible $C(A)$ fall into the this case. Again requiring $A$ and $A'$ to have matching chirality, we obtain that $M^2/12$ of all possible pairs $(A,A')$ have $(1 + \chi_0(A)\chi_0(A')) = 2 \neq 0$ and have $\rho'^{-1}\rho = ()$.  


By simultaneous conjugation, we can again assume that $\rho' = ()$, which forces $\rho = ()$. This means $C(A) = C^0$. As in the last case, we can:
\begin{itemize}
\item assume all the color $1$ edges are solid by vertex flips (on the bosons);
\item parametrize $x = E_{1,2}, y = E_{1,3}, z = E_{1,4}$ and obtain $E_{2,3} = sxy, E_{2,4} = -sxz, E_{3,4} = -syz$ for $A$, and similarly $E'_{2,3} = sx'y', E'_{2,4} = -sx'z', E'_{3,4} = -sy'z'$ for $A'$. 
\end{itemize}



We let $X = xx', Y = yy', Z = zz'$ and compute $\tr(g_{12})$: 
\[
\tr(g_{12}) = -2 xx' - 2 yy'zz' = -2X - 2YZ, 
\]
which is identical to the $(2)(1)$ case (even though the graphs are not the same, the edges of colors $1$ and $2$ are the same). Doing similar work for $g_{13}$ and $g_{14}$ and summing gives
\begin{equation*}
g(A,A') = 2(-X - Y - Z - XY - YZ - ZX) = - 2[(1+X)(1+Y)(1+Z) - 1 - XYZ].
\end{equation*}


Again by Lemma~\ref{lem:4-1-dashings-freedom}, the $X, Y, Z$ are each equally distributed among $\pm 1$. We now split our $M^2/12$ cases into subcases based on the signs of $(X,Y,Z)$:
\begin{itemize}
\item $(+, +, +)$: this happens $1/8$ of the time (total $M^2/96$), giving $g(A,A') = -2((1+1)(1+1)(1+1)-1-1) = -12$.
\item $(+,+,-), (+,-,+), (-,+,+)$: this happens $3/8$ of the time (total $M^2/32$), giving $g(A,A') = -2(0-1+1) = 0$.
\item $(+,-,-), (-,+,-), (-,-,+)$: this happens $3/8$ of the time (total $M^2/32$), giving $g(A,A') = -2(0-1-1) = 4$.
\item $(-,-,-)$: this happens $1/8$ of the time (total $M^2/96$), giving $g(A,A') = -(0-1+1) = 0$.
\end{itemize}
To summarize:
\begin{itemize}
\item $(M^2/12 + M^2/32 + M^2/96) = M^2/8$ pairs give $\mathcal{OG}(A,A') = 0$: the first $M^2/12$ term are for pairs with matched chirality; the latter $2$ terms are for pairs with matching chirality but $g(A,A') = 0$.
\item  $M^2/96$ pairs give matching chirality and $g(A,A') = -12$ (and thus $\mathcal{OG}(A,A') = 1$);
\item and $M^2/32$ pairs give $g(A,A') = 4$ and $\mathcal{OG}(A,A') = -1/3$. 
\end{itemize}



\subsection{The Final Count}

We now combine the work of the previous $2$ subsections to count instances of nonzero values of $\mathcal{OG}(A,A')$: 
\begin{itemize}
\item $1$: this only appears in the $(1)(1)(1)$ case, with $M^2/96 = 14,155,776$ pairs.
\item $1/3$: this only appears in the $(2)(1)$ case, with count $M^2/16 = 84,934,656$.
\item $-1/3$: this appears in both the $(2)(1)$ and $(1)(1)(1)$ cases, totalling $M^2/16 + M^2/32 = 3M^2/32 = 127,401,984$.
\item $0$: all $M^2/3$ cases from the $(3)$ case, $3M^2/8$ pairs from the $(2)(1)$ case, and $M^2/8$ from the $(1)(1)(1)$ case evaluate to $0$. In total, we have $3M^2/8 + M^2/3 = 5M^2/6 = 1,132,462,080$ pairs.
\end{itemize}

These numbers exactly match the empirical data in \cite{gatesjr.AdinkrasOrderedQuartets2017} that we showed in Table~\ref{table:computation}.

\section{Values of the Symmetrized Gadget}
\label{sec:jordan}
Having addressed the values of the original gadget, we now compute the values of the symmetrized gadget. Recall that, given two adinkras $A$ and $A'$, the symmetrized gadget is
\begin{eqnarray*}
	\mathcal{SG}(A,A'):=\gen{X_A,X_{A'}}&=&\sum_{I_1<J_1,I_2<J_2}\tr_{|_F}(\ophi_{I_1J_1}\ophi'_{I_2J_2})\\
	&=&\sum_{I_1<J_1,I_2<J_2}\tr_{|_F}(\nu_{I_1J_1,I_2J_2}\pi_{I_1J_1}\pi'_{I_2J_2}).
\end{eqnarray*}

We begin our analysis by calculating the value $\mathcal{SG}(A,A)$.
\begin{lem}\label{lem:square-trace} 	For each adinkra $A$, we have
	$$\tr(X_A \cdot X_A)=\tr_{|_F}\left((\ophi_{12}+\ophi_{34})^2 + (\ophi_{13}+\ophi_{24})^2 + (\ophi_{14}+\ophi_{23})^2\right).$$
\end{lem}
\begin{proof}
By definition, we have
$$X_A \cdot X_A=\sum_{I<J,K<L}\ophi_{IJ}\ophi_{KL}.$$
If $|\{I,J\}\cap\{K,L\}|=1$, then $\pi_{IJ}\neq\pi_{KL}$, from which it follows that $\pi_{IJ}\ophi_{KL}\neq 1\in K$. 
Therefore, the permutation $\pi_{IJ}\pi_{KL}$ has no fixed points and the operator $\ophi_{IJ}\ophi_{KL}$ is traceless.

Next, for each $I<J$, let $K<L$ be the complement. Then
$$(\ophi_{IJ}+\ophi_{KL})^2=\ophi_{IJ}\ophi_{IJ}+\ophi_{IJ}\ophi_{KL}+\ophi_{KL}\ophi_{IJ}+\ophi_{KL}\ophi_{KL}.$$
The above sum contains all terms in the expansion of $X_A \cdot X_A$ corresponding to the bipartition $\{1,2,3,4\}=\{I,J\}\cup\{K,L\}$. The result follows from taking the sum over all three bipartitions, and then taking the trace, restricted to the fermions.
\end{proof}

\begin{prop}\label{prop:sg_chir}
	For each adinkra $A$, we have
	$$\mathcal{SG}(A,A)=-24-8\chi_0(A)\in\{-16,-32\}.$$ 
	In particular, the value $\mathcal{SG}(A,A)$ determines the chirality class of $A$. 
\end{prop}
\begin{proof}
	
  Using Lemma \ref{lem:square-trace}, we write 
  \begin{eqnarray*}
    & & (\ophi_{12}+\ophi_{34})^2 + (\ophi_{13}+\ophi_{24})^2 + (\ophi_{14}+\ophi_{23})^2 \\
    & = & -6*\mathbbm{1}_8 + 2\ophi_{1234} + 2\ophi_{1324} + 2\ophi_{1423} \\
    & = & -6*\mathbbm{1}_8 + 2 \ophi_{1234} \\
    & = & -6*\mathbbm{1}_8 + 2 \begin{bmatrix} -\chi_0(A)\mathbbm{1}_{4} & 0 \\ 0 & \chi_0(A) \mathbbm{1}_{4}  \end{bmatrix}
  \end{eqnarray*}
using Lemma \ref{lem:signed-properties}. Restricting to fermions, $\mathcal{SG}(A,A) = \tr \left((-6 - 2\chi_0(A))\mathbbm{1}_4\right)$, which equals the desired value.

\end{proof}

\begin{remark}
	This is not true for the original gadget defined earlier: it is possible to have $\mathcal{OG}(A,A)=\mathcal{OG}(A',A')$ even though $\chi_0(A)\neq \chi_0(A')$. Thus, one benefit of the symmetrized gadget over the original gadget is that it ``detects'' chirality classes.
\end{remark}

\begin{prop}\label{prop:SG_range}
The values of $\mathcal{SG}(A,A')$ as we vary the adinkra representations consist of $\{0,\pm 16,\pm 32\}$. 
\end{prop}

\begin{proof}
The proof has the same spirit as that of Proposition \ref{prop:sg_chir}. Consider two adinkras $A$ and $A'$. The product $X_A\cdot X_{A'}$ expands as 
$$X_A\cdot X_{A'}=\sum_{I_1<J_1,I_2<J_2}\left(\ophi_{I_1J_1}\ophi'_{I_2J_2}\right)_{|_F}.$$
If $\pi_{I_1J_1}\neq\pi'_{I_2J_2}$, then the product $\ophi_{I_1J_1}\ophi'_{I_2J_2}$ is traceless since it has no entries on the main diagonal. 
This allows us to write the expansion of the gadget as follows:
\begin{eqnarray*}
	\tr(X_A\cdot X_{A'})&=&\sum_{\pi_{I_1J_1}=\pi'_{I_2J_2}}\tr_{|_F}(\ophi_{I_1J_1}\ophi'_{I_2J_2})\\
			&=&\sum_{\sigma\in\K}\bigg(\sum_{\pi_{I_1J_1}=\pi'_{I_2J_2}=\sigma}\tr_{|_F}(\ophi_{I_1J_1}\ophi'_{I_2J_2})\bigg).
\end{eqnarray*}
Let $\sigma\in\K$ and suppose that $\pi_{I_1J_1}=\pi'_{I_2J_2}=\sigma$. 
If $K_1<L_1$ and $K_2<L_2$ denote the corresponding complementary color sets, then we have
$$\pi_{I_1J_1}=\pi_{K_1L_1}=\pi'_{I_2J_2}=\pi'_{K_2L_2}=\sigma.$$
For each $\sigma\in\K$, there are precisely four terms in the above sum for which $\pi_{IJ}=\pi'_{KL}=\sigma$. 
We reorganize these four terms as follows:
$$\sum_{\pi_{I_1J_1}=\pi'_{I_2J_2}=\sigma}\tr_{|_F}(\ophi_{I_1J_1}\ophi'_{I_2J_2})=\tr_{|_F}\left((\ophi_{I_1J_1}+\ophi_{K_1L_1})(\ophi'_{I_2J_2}+\ophi'_{K_2L_2})\right).$$

We now use Corollary~\ref{cor:signed-properties} to rewrite e.g.
\[ \ophi_{I_1J_1}+\ophi_{K_1L_1} = \begin{bmatrix} (1 + s(I_1J_1K_1L_1)\chi_0(A))(\ophi_{I_1J_1})_{|_F} & 0 \\ 0 & (1- s(I_1J_1K_1L_1)\chi_0(A) (\ophi_{I_1J_1})_{|_B}\end{bmatrix}.\]
It follows that 
$$\tr_{|_F}\left((\ophi_{I_1J_1}+\ophi_{K_1L_1})(\ophi'_{I_2J_2}+\ophi'_{K_2L_2})\right) = (1 \pm \chi_0(A))(1 \pm \chi_0(A'))\tr_{|_F}(\ophi_{I_1J_1}\ophi'_{I_2J_2}), $$
where $I_1<J_1,I_2<J_2$ are two color pairs for which $\pi_{I_1J_1}=\pi'_{I_2J_2}=\sigma$, chosen arbitrarily. Furthermore, the two $\pm$'s refer to the signs $s(I_1J_1K_1L_1)$ and $s(I_2J_2K_2L_2)$, respectively. Label the above quantity $S_\sigma$. Then the symmetrized gadget can be re-written as
$$\mathcal{SG}(A,A')=\sum_{\sigma\in\K}S_\sigma = \sum_{\sigma\in\K}(1 \pm \chi_0(A))(1 \pm \chi_0(A'))\tr_{|_F}(\ophi_{\sigma}\ophi'_{\sigma}).$$


Let us now examine the value of $\tr_{|_F}(\ophi_{\sigma}\ophi'_{\sigma})$. 
If $\sigma$ sends vertex $v$ to $w$, then $\sigma$ sends $w$ to $v$; 
the composition $\ophi_{\sigma}\ophi'_{\sigma'}$ sends $v$ to $\pm v$. 
The odd dashing condition implies that if $v$ is sent to $\pm v$, then $w$ must also go to $\pm w$---that is, the signs of $v$ and $w$ agree. The other two vertices must also agree by elimination, though they can have either sign independent of $v$ and $w$. Thus, the possible values of $\tr_{|_F}(\ophi_{\sigma}\ophi'_{\sigma})$ are $\{0,\pm 4\}$, from which it follows that each $S_\sigma\in\{0,\pm 16\}$. 


Consider the sign $s(I_1J_1K_1L_1)$ in front of $\chi_0(A)$. This sign is negative for exactly one value of $\tau \in \K$ (see Corollary~\ref{cor:signed-properties}) and positive for the other $2$ values. Similarly, there is a unique group element $\tau \in \K'$ for adinkra $A'$ where the sign is negative. This leads to two cases:
\begin{itemize}
\item If $\tau\neq \tau'$, let $\sigma$ be the remaining non-trivial element of $K$. 
  Then, we have
  \begin{align*}
S_\tau  = &(1-\chi_0(A))(1+\chi_0(A'))\tr_{|_F}(\ophi_\tau\ophi'_\tau) \\
S_{\tau'}  = &(1+\chi_0(A))(1-\chi_0(A'))\tr_{|_F}(\ophi_{\tau'}\ophi'_{\tau'}) \\
S_\sigma  = &(1+\chi_0(A))(1+\chi_0(A'))\tr_{|_F}(\ophi_\sigma\ophi'_\sigma) 
  \end{align*}
Whenever e.g. $(1-\chi_0(A)) \neq 0$, we must have $(1+\chi_0(A)) = 0$. This logic shows that at most one of these $3$ terms can be nonzero, because no $2$ terms have the same sign pattern for both $\chi_0(A)$ and $\chi_0(A')$. This means $\mathcal{SG}$ must take value in $\{0,\pm 16\}$. 
\item If $\tau=\tau'$, then we have one contribution of the form $$S_\tau=(1-\chi_0(A))(1-\chi_0(A'))\tr_{|_F}(\ophi_\tau\ophi'_\tau)$$
  and $2$ contributions of the form
  $$S_\sigma=(1+\chi_0(A))(1+\chi_0(A'))\tr_{|_F}(\ophi_\sigma\ophi'_\sigma),\ \sigma\neq\tau.$$
As before, the two types of contributions cannot both be nonzero.

  If $\chi_0(A) \neq \chi_0(A')$, all the $S_\sigma$ must vanish and we get $0$ for $\mathcal{SG}$. 
If $\chi_0(A)=\chi_0(A)'=-1$, then only $S_\tau$ can be non-zero and $\mathcal{SG}$ can take values in $\{0,\pm 16\}$.
If $\chi_0(A)=\chi_0(A')=1$, then only the remaining two $S_\sigma$ for $\sigma\neq\tau$ can be non-zero, from which it follows that $\mathcal{SG}$ can take values in $\{0,\pm 16,\pm 32\}$.
\end{itemize}

\end{proof}

To confirm, our own computer numeration gives Table~\ref{table:computation_SG}, in similar spirit as Table~\ref{table:computation}. The values corroborate Proposition~\ref{prop:SG_range}.
\begin{table}[ht]
\caption{Computations from \cite{gatesjr.AdinkrasOrderedQuartets2017}.}
\begin{center}
\begin{tabular}{|c|c|}
\hline
  value of $\mathcal{SG}(A,A')$ & number of pairs $(A,A')$ \\
\hline
$32$ & $14,155,776$ \\   
$-32$ & $14,155,776$ \\   
$16$ & $198,180,864$ \\  
$-16$ & $198,180,864$ \\  
  $0$ & $934,281,216$ \\
\hline
\end{tabular}
\end{center}
\label{table:computation_SG}
\end{table}

\section{Exploring Gadgets in Higher Dimensions}
\label{sec:gadgets-general}

The overall problem of finding a ``good'' gadget for all adinkras is interesting, though ill-defined. A few things we would want from such a tool are: 
\begin{enumerate}
\item relatively few values in its range (akin to root systems having a few number of inner product values). One benefit of this feature would be to suggest good classification schemes. For example, we may be able to partition the set of adinkras into equivalence classes based on gadget values; having few values in the range would make these equivalence classes more meaningful. 
\item nice interpretations for values; for example, $\mathcal{OG}$ ensures two adinkras with different chirality classes give value $0$. Because the converse is not true, it would be good to have a gadget satisfying both directions. Also, $\mathcal{SG}$ detects the chirality class for a single adinkra (whereas $\mathcal{OG}$ does not).
\item mathematical / physical ``naturality.'' It would be nice for the concept to generalize or relate to concepts with good precedent from e.g. representation theory or physics.
\end{enumerate}

Gates has proposed one alternative to his original gadget construction in \cite{gatesAdinkrasOrderedQuartets2018}, in which a symmetric version of the gadget appears. 
This gadget is different from the symmetrized gadget introduced in this paper, but it still enjoys many good properties. 
In particular, the gadget defined in \cite{gatesAdinkrasOrderedQuartets2018} can detect chirality. 
While the techniques of our paper can similarly explain the values Gates' new gadget, one is still left with trying to find a mathematically natural foundation for the new gadget. 

In this section, we explore the problem of proposing gadgets for higher dimensions through a combination of the lens above.

\subsection{Generalizing to $(n=4,k=0)$ Adinkras}
\label{sec:computational}

Most adinkras of physical interest have $n$ greater than $4$.  The exponential growth with $n$ of the number of $(n, k)$ adinkras prohibits brute-force computation (in the manner of \cite{gatesjr.AdinkrasOrderedQuartets2017}) of $\mathcal{OG}$ between all pairs of $(n, k)$ adinkras for general $n$ and $k$.  The next smallest class of adinkras is the class of $(n=4, k=0)$ adinkras, of which there are $6,658,877,030,400$. These adinkras have $2^{4-0} = 16$ vertices and $4$ colors, with the underlying graph isomorphic to the $4$-dimensional Hamming cube. Using a more efficient algorithmic approach than that of \cite{gatesjr.AdinkrasOrderedQuartets2017}, we compute the range of $\mathcal{OG}$ for the $(4, 0)$ adinkras to be  $$(-1/24)\{0, \pm 2, \pm 4, \pm 6, \pm 8, \pm 10, \pm 12, -14, \pm 16, -18, -24, -28, -48\}.$$
We emphasize that the negative signs, as opposed to $\pm$, on some of the entries (especially $14$) are \textbf{not} typos!

That the range of $(4, 0)$ adinkra gadgets does not admit an obvious geometric interpretation, combined with the fact that $(4, 0)$ adinkras, unlike $(4, 1)$ adinkras, depict reducible supermultiplets, suggests that in addition to distinguishing adinkras within the same $(n, k)$ class, perhaps gadgets can capture, with their ranges, qualitative distinctions between $(n, k)$ classes.  One may e.g. wish to define a gadget whose range possesses a certain mathematical property precisely when it is computed over a class of adinkras with a certain physically meaningful property.

No other proposed gadgets of the same general form as $\mathcal{OG}$, i.e. sums of traces of products, have yielded more manageable ranges for the $(4, 0)$ adinkras. This includes the symmetrized gadget $\mathcal{SG}$, which we compute to have the range
$$\{i | i \in 2\ZZ, -34 \leq i \leq 34\} \cup \{\pm 40, \pm 48\}.$$
We plan to include details about the algorithmic and computational aspects of this problem in upcoming works.

\subsection{Representation-theoretic Interpretations}
\label{sec:representation-theoretical}

There is a representation-theoretic interpretation of the two gadgets that sheds more ``philosophical light'' on the phenomena they exhibit, using the language of group characters of tensor products of representations.

Given a group $G$ and two $G$-modules $V$ and $W$ with characters $\chi_V$ and $\chi_W$ respectively, there are two ``natural'' actions on $V \otimes W$ related to $G$:
\begin{itemize}
\item The first is a $G \times G$ action where $$(g_1, g_2)(v \otimes w) = (g_1v \otimes g_2w).$$
  We denote \emph{outer product module} by $V \otimes W$. It has character $$\chi_{V \otimes W}(g_1,g_2) = \tr(V(g_1))\tr(W(g_2)).$$ Given a character, a natural object is the sum of the character over the group, so in the first case, we would get the object $$ \sum_{(g_1, g_2) \in G \times G} \tr(V(g_1))\tr(W(g_2)) = \bigl(\sum_{g \in G} \tr(V(g))\bigr)\bigl(\sum_{g \in G}\tr(W(g)))\bigr).$$ To compare, our symmetrized gadget is doing something like $$\sum_{(g_1, g_2) \in G \times G} \tr(V(g_1)W(g_2))$$
\item The second is a $G$ action where $$g(v \otimes w) = (gv \otimes gw).$$ We denote this \emph{inner product module} by $V \hat{\otimes} W$. It has character $$\chi_{V \hat{\otimes} W}(g) = \tr(V(g))\tr(W(g)).$$  Summing over the group, we would get the object $$\sum_{g \in G} \tr(V(g))\tr(W(g)).$$ To compare, the (non-symmetrized) gadget is something like $$\sum_{g \in G} \tr(V(g)W(g)).$$

\end{itemize}

To summarize, both proposed gadgets are similar to very natural objects coming from representation theory with one very important difference; namely, the traces are taken \textbf{after} the multiplication, which makes the objects seem less natural. Thus, it is somewhat surprising that the gadgets have nice properties at all, if one believes that nice mathematical results always happen for natural reasons. One interpretation is that we are simply lucky --- that is, our dimensions are very small and our terms ended up coincidentally close to a Klein-$4$ group, as another consequence of ``The Law of Small Numbers.''

\section{Conclusion}
\label{sec:conclusion}

In our work, we have examined some numerical phenomena that appeared for gadgets from the study of adinkras and garden algebras. We gave some theoretical reasoning backing up these observations and suggested another definition of gadgets that might be useful, ending with a top-down view of what happens when we try to generalize our gadgets to higher dimensions. We are motivated by the desire to simplify and clarify; by trying to find parsimonious reasons why the gadget takes so few values, we hope to understand better the mathematical structures behind a mysterious-looking object.

Fruitful future work would include physical justification of gadgets definitions (old or new), hopefully inspired by independently interesting computations that arise in higher dimensions, or by independently rich mathematical/physical theory.

\section*{Acknowledgments}

We thank S. James Gates, Jr. for introducing his concept of the gadget to us and Kevin Iga for extensive conversation about this problem. We would further like to acknowledge our participation in the annual Brown University Adinkra Math/Phys ``Hangouts'' in 2017. I. Friend thanks Adam Artymowicz for his library of NumPy functions, which eased the creation of the permutation matrices used in the computation of the range of $\mathcal{OG}$ for the $(n=4,k=0)$ adinkra class.  

\bibliographystyle{abbrv}
\bibliography{gadgets}

\begin{thebibliography}{10}

\bibitem{calkinsAdinkras0branesHoloraumy2015}
M.~Calkins, D.~E.~A. Gates, S.~J. Gates~Jr., and K.~Stiffler.
\newblock Adinkras, 0-branes, {{Holoraumy}} and the {{SUSY QFT}}/{{QM
  Correspondence}}.
\newblock {\em International Journal of Modern Physics A}, 30(11):1550050, Apr.
  2015.

\bibitem{douglas}
B.~B. Douglas, S.~G. Jr, S.~{Gates Jr.}, and J.~Wang.
\newblock {Automorphism Properties of Adinkras}.
\newblock \url{arXiv:hep-th/1009.1449}, Sept. 2010.

\bibitem{douglasAutomorphismPropertiesClassification2015}
B.~L. Douglas, S.~J. Gates, B.~L. Segler, and J.~B. Wang.
\newblock Automorphism {{Properties}} and {{Classification}} of {{Adinkras}}.
\newblock {\em Advances in Mathematical Physics}, 2015:1--17, 2015.

\bibitem{d2l:first}
M.~Faux and S.~{Gates Jr}.
\newblock {Adinkras: A graphical technology for supersymmetric representation
  theory}.
\newblock {\em Physical Review D}, 71(6), 2005.

\bibitem{gatesLorentzCovariantHoloraumyInduced2015}
J.~Gates, T.~Grover, M.~D. {Miller-Dickson}, B.~A. Mondal, A.~Oskoui, S.~Regmi,
  E.~Ross, and R.~Shetty.
\newblock A {{Lorentz Covariant Holoraumy}}-{{Induced}} "{{Gadget}}" {{From
  Minimal Off}}-{{Shell 4D}}, {{N}} = 1 {{Supermultiplets}}.
\newblock {\em Journal of High Energy Physics}, 2015(11), Nov. 2015.

\bibitem{gatesAdinkrasOrderedQuartets2018}
J.~Gates, L.~Kang, D.~S. Kessler, and V.~Korotkikh.
\newblock Adinkras {{From Ordered Quartets}} of {{BC4 Coxeter Group Elements}}
  and {{Regarding Another Gadget}}'s 1,358,954,496 {{Matrix Elements}}.
\newblock {\em International Journal of Modern Physics A}, 33(12):1850066, Apr.
  2018.

\bibitem{gates:genomics}
S.~J. {Gates Jr.}, J.~Gonzales, B.~{Mac Gregor}, J.~Parker, R.~Polo-Sherk,
  V.~Rodgers, and L.~Wassink.
\newblock {4D, $\mathcal{N} = 1$ supersymmetry genomics (I)}.
\newblock {\em Journal of High Energy Physics}, 12, Dec. 2009.

\bibitem{gatesjr.AdinkrasOrderedQuartets2017}
S.~J. Gates~Jr., F.~Guyton, S.~Harmalkar, D.~S. Kessler, V.~Korotkikh, and
  V.~A. Meszaros.
\newblock Adinkras {{From Ordered Quartets}} of {$BC_4$ {{Coxeter Group
  Elements}} and {{Regarding}} 1,358,954,496 {{Matrix Elements}} of the
  {{Gadget}}}.
\newblock {\em Journal of High Energy Physics}, 2017(6), June 2017.

\bibitem{zhang:adinkras}
Y.~X. Zhang.
\newblock Adinkras for mathematicians.
\newblock {\em Transactions of the American Mathematical Society},
  366(6):3325--3355, 2014.

\bibitem{zhang:counting}
Y.~X. Zhang.
\newblock {A Unified Enumeration of 1-dimension Garden Algebras and Valise
  Adinkras}.
\newblock \url{http://arxiv.org/abs/1801.02678}, 2018.

\end{thebibliography}


\end{document}